\documentclass[11pt]{article}

\makeatletter
\DeclareRobustCommand{\qed}{%
  \ifmmode 
  \else \leavevmode\unskip\penalty9999 \hbox{}\nobreak\hfill
  \fi
  \quad\hbox{\qedsymbol}}
\newcommand{\openbox}{\leavevmode
  \hbox to.77778em{%
  \hfil\vrule
  \vbox to.675em{\hrule width.6em\vfil\hrule}%
  \vrule\hfil}}
\newcommand{\qedsymbol}{\openbox}
\newenvironment{proof}[1][\proofname]{\par
  \normalfont
  \topsep6\p@\@plus6\p@ \trivlist
 \item[\hskip\labelsep\itshape
    #1.]\ignorespaces
}{%
  \qed\endtrivlist
}
\newcommand{\proofname}{Proof}
\makeatother

\usepackage[utf8]{inputenc}   
\usepackage{bm}

\usepackage{color}
\usepackage{latexsym}
\usepackage{dsfont}
\usepackage{amssymb}
\usepackage{comment}
\usepackage{graphicx}
\usepackage{amsmath,amsfonts,amssymb,theorem,euscript,array,enumerate,amsfonts,mathrsfs}
\usepackage{hyperref}
\usepackage{appendix}
\usepackage{caption}
\usepackage[T1]{fontenc}
\usepackage{babel}
\numberwithin{equation}{section}
\usepackage{bbm}
\usepackage{subfigure}
\usepackage{color}
\usepackage{stmaryrd}

\usepackage[algo2e,ruled,vlined]{algorithm2e} 
 \usepackage{algorithm}
\usepackage{algorithmicx}
\usepackage{ marvosym }
\usepackage{hyperref}
\usepackage{bm}
\usepackage{xcolor, soul}

\usepackage{mathtools}
\mathtoolsset{showonlyrefs}

\usepackage{comment}

\usepackage{tikz}
\usetikzlibrary{fit,matrix,chains,positioning,decorations.pathreplacing,arrows}

\usepackage{mathtools}

\usepackage{geometry}
\usepackage[normalem]{ulem}

\newcommand{\pr}{\text{pr}}

\def \b1{\bf{1}}

\def \N{\mathbb{N}}
\def \R{\mathbb{R}}

\def \M{\mathbb{M}}

\def \E{\mathbb{E}}
\def \F{\mathbb{F}}

\def \P{\mathbb{P}}

\def \S{\mathbb{S}}
\def \W{\mathbb{W}}

\def \mrH{\mathrm{H}}

\def \d{\mathrm{d}}

\def\esssup_#1{\underset{#1}{\mathrm{ess\,sup\, }}}

\def\argmin_#1{\underset{#1}{\mathrm{argmin\, }}}
\def\argmax_#1{\underset{#1}{\mathrm{argmax\, }}}

\def\dm#1{\frac{\delta}{\delta m}}

\def \Ac{{\cal A}}
\def \Bc{{\cal B}}
\def \Cc{{\cal C}}

\def \Ec{{\cal E}}
\def \Fc{{\cal F}}

\def \Ic{{\cal I}}
\def \Kc{{\cal K}}

\def \Pc{{\cal P}}

\def \Mc{{\cal M}}
\def \Nc{{\cal N}}

\def \Sc{{\cal S}}
\def \Tc{{\cal T}}
\def \Uc{{\cal U}}
\def \Vc{{\cal V}}
\def \Wc{{\cal W}}

\def \d{\mathrm{d}}

\def\beqs{\begin{eqnarray*}}
\def\enqs{\end{eqnarray*}}
\def\beq{\begin{eqnarray}}
\def\enq{\end{eqnarray}}

\newcommand{\red}[1]{\textcolor{red}{#1}}

\def\rd#1{{\color{red}#1}}

\def\red#1{{\color{red}#1}}
\def\ora#1{{\color{orange}#1}}

\newtheorem{Theorem}{Theorem}[section] 
\newtheorem{Definition}[Theorem]{Definition} 
\newtheorem{Proposition}[Theorem]{Proposition}
\newtheorem{Assumption}[Theorem]{Assumption}
\newtheorem{Lemma}[Theorem]{Lemma}

\newtheorem{Remark}[Theorem]{Remark}

\title{
Learning operators on labelled conditional distributions with applications to mean field control of non exchangeable systems
}

\author{
Samy MEKKAOUI\footnote{École Polytechnique, CMAP, 
\sf samy.mekkaoui at polytechnique.edu. 
This author is supported by the S-G Chair ``Financial Risks'' and the Qube-RT Chair ``Deep Finance and Statistics'', samy.mekkaoui at polytechnique.edu} 
\and 
Huy\^en PHAM\footnote{École Polytechnique, CMAP, 
\sf huyen.pham at polytechnique.edu. 
This author is supported by the BNP Paribas Chair ``Futures of Quantitative Finance'', the Chair ``Financial Risks'', by FiME (Laboratory of Finance and Energy Markets), and the EDF–CACIB Chair ``Finance and Sustainable Development'', huyen.pham at polytechnique.edu}
\and
Xavier WARIN \footnote{EDF R\&D, and FiME, xavier.warin at edf.fr }
}

\begin{document}

\maketitle
\begin{abstract}
We study the approximation of operators acting on probability measures on a product space with prescribed marginal. 
Let $I$ be a label space endowed with a reference measure $\lambda$, and define $\Mc_\lambda$ as the set of probability measures on $I\times\R^d$ with first marginal $\lambda$.  By disintegration, elements of $\Mc_\lambda$ correspond to families of labeled conditional distributions. Operators defined on this constrained measure space arise naturally in mean-field control problems with heterogeneous, non-exchangeable agents.

Our main theoretical result establishes a universal approximation theorem for continuous operators on $\Mc_\lambda$. 
The proof combines cylindrical approximations of probability measures with a DeepONet-type branch–trunk neural architecture, yielding finite-dimensional representations of such operators. We further introduce a sampling strategy for generating training measures in $\Mc_\lambda$, enabling practical learning of such conditional mean-field operators.  
 
We apply the method to the numerical resolution of mean-field control problems with heterogeneous interactions, thereby extending previous neural approaches developed for  homo\-geneous (exchangeable)  systems. Numerical experiments illustrate the accuracy and computa\-tional effectiveness of the proposed framework.

\end{abstract}

\vspace{5mm}


\vspace{5mm}



\section{Introduction}

\paragraph*{Motivation of MFC with non exchangeable systems.}
Mean-field control provides a tractable framework for modeling and optimizing large interacting particle systems. In the classical setting, exchangeability  of agents allows the population to be described by a single marginal distribution, representing the law of a representative agent, and the associated control problems depend only on this evolving law. This structure underlies standard formulations of mean-field games and mean-field control, see \cite{benetal13}
\cite{carmona_probabilistic_2018a}, \cite{carmona_probabilistic_2018b}. 

In many applications, however, agents are heterogeneous and cannot be treated as exchangeable. Agents may differ through labels, types, spatial indices, or intrinsic parameters, and 
heterogeneous interactions are often modeled through graphons or related structures; see e.g. \cite{bayraktar2023graphonAAP}, \cite{caihua21SICON}, \cite{coppini2024nonlinearSPA}, \cite{jabin2025mean}. In such non-exchangeable systems, the population is no longer characterized by a single marginal distribution. Instead, it is described by a family of probability distributions indexed by a label space. Equivalently, the population law is a probability measure on a product space with prescribed first marginal, see \cite{lacker2023label}.  
This reformulation leads naturally to operators acting on constrained measure spaces rather than on a single probability distribution.

The efficient approximation of such operators is the central objective of this work. 
More precisely, hetero\-geneous mean-field control problems lead to nonlinear mappings that associate to a population law its corres\-ponding value function, decoupling field, or feedback control. These mappings act on families of conditional distributions and therefore on probability measures on a product space with prescribed marginal. We formulate this approximation problem in a functional-analytic framework and develop a neural operator methodology adapted to this structure.

\paragraph*{Operators on labelled conditional distributions.} Let $I$ be a compact label space, say $I$ $=$ $[0,1]$, endowed with a reference probability measure, e.g.,  the uniform distribution. We consider the constrained space 
\begin{align}
\Mc_\lambda & := \; \big\{ \mu \in \Pc_2(I\times\R^d): \pr_1\sharp \mu = \lambda \big\},     
\end{align}
where $\Pc_2(I\times\R^d)$ is the Wasserstein space of square integrable  probability measures on $I\times\R^d$, $I \times \R^d \ni (u,x) \mapsto \pr_{1}(u,x)=u$ denotes the projection mapping on the first coordinate and $\sharp$ denotes the pushforward measure. By disintegration, each $\mu$ $\in$ $\Mc_\lambda$  can be written as 
\begin{align}
\mu(\d u,\d x) &= \; \lambda(\d u) \mu^u(\d x),  
\end{align}
where $(\mu^u)_{u\in I}$ is a $\lambda(\d u)$-a.e unique family of labelled conditional distributions, valued in
\begin{align}
    L^2_{\lambda}(I ; \Pc_2(\R^d)) := \big \lbrace I \ni u \mapsto \mu^u \in \Pc_2(\R^d) \text{ measurable and } \int_{I} \Wc_2(\mu^u,\delta_0) \lambda(\d u ) < \infty \big \rbrace \red{,}
\end{align}
where $\Wc_2$ denotes the $2$-Wasserstein distance.
We study operators of the form 
\begin{align} \label{operator}
\vartheta \; : \; \Mc_\lambda  \; \ni \; \mu \;  \longmapsto \; V(\cdot,\cdot,\mu) \in L^2(\mu),  
\end{align}
for some function $V$ on $I\times\R^d\times\Mc_\lambda$ with quadratic growth w.r.t. the first two arguments, and $L^2(\mu)$ is the space of square integrable functions w.r.t. $\mu$ $\in$ $\Mc_\lambda$.  
Such operators arise naturally in heterogeneous mean-field control problems, for instance as decoupling fields associated with forward–backward stochastic systems in maximum principle, see \cite{kharroubi2025stochastic}, \cite{cao2025graphon}, or feedback maps obtained from dynamic programming, see \cite{decrescenzo2024mean}. 

Unlike classical neural operator learning, the domain here is an infinite-dimensional space of measures subject to a marginal constraint. Preserving this structural constraint is essential both for theoretical consistency and for numerical stability.

\paragraph*{Our main contributions.} The main contributions of this work are fourfold.
\begin{enumerate}
\item First, we introduce a neural operator framework tailored to operators defined on the constrained Wasserstein space 
$\Mc_\lambda$. The construction combines cylindrical approximations of probability measures with a DeepONet-type branch–trunk architecture, yielding finite-dimensional repre\-sentations that are compatible with the marginal constraint 
$\pr_1\sharp \mu = \lambda$. 
\item Second, we establish a universal approximation theorem for continuous operators $\vartheta$ as in \eqref{operator}. More precisely, we prove that the proposed architecture is dense in the class of such operators with respect to the natural topology induced by the Wasserstein distance. The proof integrates approximation results for probability measures with neural operator theory and shows that the marginal constraint can be preserved at the approximation level.  
\item Third, we develop a constructive sampling procedure for generating training measures in $\Mc_\lambda$. 
The method enforces the prescribed marginal on the label space while producing diverse families of conditional laws, thereby enabling practical training of conditional mean-field neural operators in a manner consistent with the theoretical framework. 
\item Finally, we apply the proposed methodology to the numerical resolution of mean-field control problems with non-exchangeable interactions. Relying on maximum principle and dynamic programming formu\-lations developed for heterogeneous systems, we approximate decoupling fields, value functions, and feedback controls by conditional mean-field neural operators and use them to solve the associated forward-backward stochastic differential equations and HJB equations. This extends neural mean-field control methods previously developed for homogeneous (exchangeable) systems to the hetero\-geneous setting. 
\end{enumerate}

\paragraph*{Related work.} The present work lies at the intersection of neural operator learning, approximation on probability measure spaces, and mean-field control. 

Neural operator architectures aim at approximating mappings between infinite-dimensional spaces and have been successfully applied to the numerical solution of partial differential equations. DeepONet \cite{lu2021learning}, \cite{lanthaler2022error} provide  universal approximation results for nonlinear operators between Banach spaces. 
Related operator-learning frameworks include Fourier neural operators \cite{li2021fourier} and subsequent developments in operator approximation theory. These approaches, however, are formulated for operators acting on function spaces. In contrast, we consider operators defined on the constrained Wasserstein space $\Mc_\lambda$, whose elements are probability measures on a product domain with prescribed marginal. Extending operator-learning techniques to this setting requires combining neural operator theory with measure-theoretic representations compatible with the marginal constraint.  

Learning and approximation on spaces of probability measures have been studied through permutation-invariant architectures such as DeepSets \cite{zaheer2017deepsets}, see \cite{gerlauphawar22}. 
Cylindrical approximations of measures and universal approximation results on Wasserstein spaces have been investigated in \cite{cucetal19}, 
\cite{guophaweiSPA} and \cite{phamwarin23}. These approaches typically address functions defined on unconstrained spaces of measures.

Numerical methods for mean-field control and mean-field games include PDE or probabilistic schemes and neural network approaches for solving forward–backward systems and HJB equations; see, e.g., \cite{ruttoto},  
\cite{carlau22}, \cite{pham2022mean}, \cite{reistozha}. In the non-exchangeable settings, the associated optimality systems involve operators acting on conditional distributions rather than on a single marginal law. To the best of our knowledge, neural operator approximations for such operators on constrained measure spaces arising in heterogeneous mean-field control have not been investigated.

\paragraph*{Outline of the paper.} The remainder of the paper is organized as follows. In Section \ref{sec:operator}, we introduce the functional framework for operators defined on the constrained Wasserstein space $\Mc_\lambda$, present the conditional mean-field neural operator architecture, and establish the universal approximation theorem. We also describe the sampling procedure for generating training measures in $\Mc_\lambda$ and the associated training methodology. 
Section \ref{sec:num} provides numerical experiments illustrating the approximation properties of the proposed neural operators for representative conditional mean-field functionals. In Section \ref{sec:appli}, 
we apply the method to the numerical resolution of mean-field control problems with non-exchangeable interactions. We describe algorithms based on maximum principle and dynamic programming formulations and approximate the associated decoupling fields, value functions, and feedback controls using conditional mean-field neural operators.

\paragraph*{Notations.}\label{sec: notations}
\begin{enumerate} 
    \item [$\bullet$] We denote by $\Pc_2(\R^d)$ the Wasserstein space of square integrable probability measures equipped with the 2-Wasserstein distance $\Wc_2$. 
    Given $\mu \in \Pc_2(\R^d)$, we denote by $L^2(\mu)$ the space of measurable functions on $\R^d$ s.t
    \begin{align}
       | \phi|_{\mu}^2 := \int_{\R^d} |\phi(x)|^2 \mu(\d x) < + \infty.
    \end{align}
    Given vector-valued maps $f,g \in L^2(\mu)$, we denote  $\langle f, g \rangle_{L^2(\mu)}:=  \int_{\R^d} f(x) \cdot g(x) \mu(\d x)$ as their inner product where $\cdot$ is the inner product between vectors. 
   Given $\mu \in \Pc_{2}(\R^d)$, and $\phi \in L^2(\mu)$, we set $\E_{X \sim \mu} [\phi(X) ] := \int_{\R^d} \phi(x) \mu(\d x)$, where $X$ denotes a random variable defined on some  probability space $(\Omega,\Fc,\P)$. We denote by $\P_{X}$ the law of $X$ under $\P$.  Given a measurable map $\phi : \R^d \to \R^k$ and a measure $\mu \in \Pc_2(\R^d)$, we denote by $\langle \phi,\mu \rangle := \int_{\R^d} \phi(x) \mu(\d x)$.
   \item [$\bullet$] Given an horizon time $T > 0$ and a normed vector space $(E, \lVert \cdot \Vert_E)$, we denote by $\Cc([t,T]; E)$ the space of continuous maps from $[t,T]$ into $E$ endowed with Borel $\sigma$-algebra and its supremum norm $\lVert \omega \rVert_{\Cc[t,T];E)} := \underset{t \leq s \leq T}{\text{ sup }} \lVert \omega_s \rVert_E$. 
    When $E = \R^d$, we will only write $\Cc^d_{[t,T]}$ and when $d=1$, only $\Cc_{[t,T]}$. We denote by $\W_T$ the Wiener measure on $\Cc^d_{[0,T]}$. We also denote by $\Cc(E)$ the space of continuous functions on $E$ into $\R$.
\end{enumerate}

\section{Operator learning of conditional mean field functionals} \label{sec:operator}

Given a function $V$ on $I \times \R^d  \times \Mc_{\lambda}$ valued in $\R^p$, with quadratic growth condition w.r.t the first two arguments, we aim to approximate the infinite dimensional map
\begin{align}
    \Vc : \mu \in \Mc_{\lambda} \mapsto V(\cdot,\cdot,\mu) \in L^2(\mu),
\end{align}
called non exchangeable mean field function, by a map $\Nc$ constructed by some combinations of neural networks. The mean-field neural network $\Nc$ takes input of two parts: $\mu$ a probability measure on $\Mc_{\lambda}$ and $(u,x)$ in the support of $\mu$ and outputs $\Nc(\mu)(u,x)$. The quality of this approximation is measured by the error
 \begin{align}
        L(\Nc) := \int_{\Mc_{\lambda}} \Ec_{\Nc}(\mu) \rho(\d \mu),
\end{align}
with 
    \begin{align}
        \Ec_{\Nc}(\mu) : = | \Vc(\mu) - \Nc(\mu)|^2_{\mu} = \E_{(U,X) \sim \mu} \Big[ | V(U,X,\mu) - \Nc(\mu)(U,X) |^2 \Big],
    \end{align}
and where $\rho$ is a probability measure over the Borel space $\Mc_{\lambda}$, called training measure. The learning of the mean-field functional $\Nc$ will then be performed by minimizing over the parameters of the neural network $\Nc$ the loss function
\begin{align}
L_M(\Nc) := \frac{1}{M} \sum_{m=1}^{M} \Ec_N(\mu^{(m)}),
\end{align}
where $\mu^{(m)}$, $m \in \llbracket 1,M \rrbracket$ are training samples of $\rho$.

\subsection{Neural network approximation}

We state a universal approximation theorem which will rely on the use of DeepONet architecture,  the cylindrical structure and their approximation results on Wasserstein space used to learn operators on $\Mc_{\lambda}$. The universal approximation theorem is stated with an $L^2$-distance, which is the one used in practice during the training process. 

\begin{Theorem}\label{thm : universal_approximation_theorem}
    Let $\rho$ be a probability measure on the Borel space $\Mc_{\lambda}$ and $V$ a continuous map from $I \times \R^d \times \Mc_{\lambda}$ into $\R^q$ such that $\lVert V \rVert^2_{L^2(\rho)} := \int_{\Mc_{\lambda}} |V(\cdot,\cdot,\mu)|^2_{\mu} \rho(\d \mu) < \infty$. Then, for all $\epsilon > 0$, there exists $J,r \in \N^{\star}$, maps $\varphi_1,\ldots,\varphi_J \in \Cc(I \times \R^d)$, trunk nets $(\Tc_k)_{1 \leq k \leq r}$ from $I \times \R^d$ into $\R$ and  branch nets $(\Bc_k)_{1 \leq k \leq r}$ from $\R^J \to \R^q$ such that  
    \begin{align}
        \int_{\Mc_{\lambda}} \E_{(U,X) \sim \mu} \Big[ \big| V(U,X,\mu) -  \sum_{k=1}^{r} \Tc_k(U,X) \Bc_{k}(  \Phi_J(\mu))     \big|^2 \Big] \rho(\d \mu) \leq \epsilon,
    \end{align}
     where $\Phi_J(\mu) := \big( \langle \varphi_1,\mu \rangle, \ldots, \langle \varphi_J,\mu \rangle \big) \in \R^J$.
\end{Theorem}

\begin{Remark}

The architecture of Theorem \ref{thm : universal_approximation_theorem}, i.e.
\begin{align}
    I \times \R^d \times \Mc_{\lambda} \ni (u,x,\mu) \mapsto \sum_{k=1}^{r} \Tc_k(u,x) \Bc_k(\Phi_J(\mu)),
\end{align}
is called DeepONetCyl. In addition to this architecture, we could also consider the more general class of neural network architectures, 
\begin{align}\label{eq:neural_net_architecture_2}
I \times \R^d \times \Mc_{\lambda} \ni (u,x,\mu) \mapsto 
\Psi_{\theta}\Big(u,x, \Phi_J(\mu)\Big),
\end{align}
where $\Psi_{\theta} : I \times \R^d \times \R^J \to \R^q$. Since this architecture produces identical results, we only present those obtained with DeepONetCyl in the sequel.
\end{Remark}


\subsection{Proof of the universal approximation theorem}

Let $\epsilon > 0$. Fix $\Kc$ a compact subset of $I \times \R^d$. For simplicity, we suppose $V$ is a $\R$-valued map but the proof can be easily extended to a vector-valued function. Therefore, let  $V$ be a continuous function over $I \times \R^d \times \Mc_{\lambda}$ into $\R$.

\noindent \textbf{Step n°1 : Separability for a dense class of functions}

Since $\Kc$ is a compact metric space, the space $\Cc(\Kc)$ is separable for the uniform norm and therefore, there exists a dense family $(\varphi_n)_{n \in \N} \in \Cc(\Kc)$ for $\lVert \cdot \rVert_{\infty;\Kc} := \underset{x \in \Kc}{\text{ sup }} |\cdot | $. Then, the family of maps $\big \lbrace \Mc_{\lambda}(\Kc) \ni \mu \mapsto \langle  \varphi_n, \mu \rangle \in \R :  n \in \N \big \rbrace $   where  $\Mc_{\lambda}(\Kc) := \big \lbrace \mu \in \Pc(\Kc ) : \text{pr}_1 \sharp \mu = \lambda \big \rbrace $ is such that for any $\mu \neq \nu \in \Mc_{\lambda}(\Kc)$, there exists $j \in \N$ s.t $\langle \varphi_j, \mu \rangle \neq \langle \varphi_j, \nu \rangle$.

    Indeed, let $\mu,\nu \in \Mc_{\lambda}(\Kc)$ such that $\mu \neq \nu$ and let $f \in \Cc(\Kc)$ such that $\langle f , \mu \rangle \neq \langle f , \nu \rangle$ (such map $f$ exists by characterization of Borel measures on the compact metric space $\Kc$ by bounded continuous maps) and let $\delta := | \int_{\Kc} f(x) \d \mu(x) - \int_{\Kc} f(x) \d \nu(x) \big| > 0 $ and $j_0 \in \N$ such that $ \lVert \varphi_{j_0} - f \lVert_{\infty; \Kc} < \frac{\delta}{4}$. Therefore, we have
    \begin{align}
        \Big | \int_{\Kc} \varphi_{j_0} \d \mu - \int_{\Kc} \varphi_{j_0} \d \nu \Big | &= \Big| \int_{\Kc} f \d \mu - \int_{\Kc} f \d \nu  + \int_{\Kc} (\varphi_{j_0}-f) \d \mu   - \int_{\Kc} (\varphi_{j_0}-f) \d \nu \Big|  \\
        &\geq \big | \int_{\Kc} f \d \mu - \int_{\Kc} f \d \nu \big| - \Big| \int_{\Kc} (\varphi_{j_0}-f) \d \mu   - \int_{\Kc} (\varphi_{j_0}-f) \d \nu \Big|  \\
        & \geq  \delta - 2 \lVert \varphi_{j_0} - f \rVert_{\infty;\Kc} > 0
    \end{align} 
    For $J \in \N^{\star}$, we define
    \begin{align}
       \Mc_{\lambda}(\Kc) \ni \mu \mapsto  \Phi_J(\mu) := \big( \langle \varphi_1,\mu \rangle, \ldots, \langle \varphi_J,\mu \rangle \big) \in \R^J.
    \end{align}
    By definition of the weak convergence of probability measures and since $\Kc$ is compact,  $\Phi_J$ is clearly continuous in the topology of the Wasserstein distance.

\noindent \textbf{Step n°2 : Construction of a dense sub-algebra of $\Cc(\Kc \times \Mc_{\lambda}(\Kc))$}    

We now denote the class of maps $\Ac$ on $\Kc \times \Mc_{\lambda}(\Kc)$ into $\R$ as
    \begin{align}
\Ac :=\Big\{ F : \mathcal K \times \Mc_{\lambda}(\mathcal K) \to \mathbb R
\ \Big|\
\exists\, J,r \in \mathbb N^{\ast},
\ (f_k)_{1\le k\le r} \subset \mathcal C(\mathbb R^J),\
(g_k)_{1\le k\le r} \subset \mathcal C(\mathcal K)
\\
\text{such that  }
F(u,x,\mu)
=
\sum_{k=1}^{r}
f_k\big(\Phi_J(\mu)\big)\, g_k(u,x)
\Big\}.
\end{align}
We claim that given the supremum norm over the compact space $\Kc \times \Mc_{\lambda}(\Kc)$, the closure $\bar{\Ac} = \Cc(\Kc \times \Pc_2^{\lambda}(\Kc))$.

Following Stone-Weierstrass, it is sufficient to show that $\Ac$ is a sub-algebra of $\Cc(\Kc \times \Mc_{\lambda}(\Kc))$ which containts constant maps and where for any $(k_1,\mu_1) \neq (k_2,\mu_2)$ where $k_1,k_2 \in \Kc$ and $\mu_1,\mu_2 \in \Mc_{\lambda}(\Kc)$, there exists $F \in \Ac$ such that $F(k_1,\mu_1) \neq F(k_2,\mu_2)$.

\begin{enumerate}
    \item [$\bullet$] The stability by addition and multiplication are clear.
    \item [$\bullet$] It is clear that $\Ac$ contains constant maps by taking $J,r=1$, $f_1=1 \in \Cc(\R)$ and $g_1=1 \in \Cc(\Kc)$.
    \item [$\bullet$] Let $k_1 \neq k_2$. In this case, take $J=1,r=1$, $f_1=1$ and $g_1(k) = d(k,k_1)$ for which it is clear that $F(k_1,\mu_1)= g_1(k_1)= 0 \neq g_1(k_2) = F(k_2,\mu_2)$ where $d\big( (u_1,x_1),(u_2,x_2)\big) = |u_1-u_2| + | x_1 - x_2 |$ is the product distance. Suppose now $k_1=k_2$ and $\mu_1 \neq \mu_2$. By  the previous result, there exists $j_0 \in \N^{\star}$ such that $\langle \varphi_{j_0},\mu_1 \rangle \neq \langle \varphi_{j_0},\mu_2 \rangle$. Then, we take $J=j_0$, $r=1$, $g=1$ and $f_1 :(x_1,\ldots,x_{j_0}) \mapsto x_{j_0} \in \Cc(\R^{j_0})$. In this case, we have  $F(k_1,\mu_1) = \langle \varphi_{j_0}, \mu_1 \rangle \neq \langle \varphi_{j_0},\mu_2 \rangle = F(k_2,\mu_2)$.
\end{enumerate}
Therefore by Stone-Weierstrass, $\Ac$ is a dense sub-algebra of $\Cc \big(\Kc \times \Mc_{\lambda}(\Kc) \big)$. Therefore, for any $\epsilon > 0$, there exists $J ,r \in \N^{\star}$, $(f_{k})_{1 \leq k \leq r} \in \Cc(\R^J)$ and $(g_{k})_{1 \leq k \leq r} \in \Cc(\Kc)$ such that
\begin{align}
    \underset{(u,x,\mu) \in \Kc \times \Mc_{\lambda}(\Kc)}{\text{ sup }} |V(u,x,\mu) - \sum_{k=1}^r f_{k}( \Phi_{J}(\mu))) g_{k}(u,x) | \leq \epsilon.
\end{align}
Now, noticing that $\tilde{\Kc}= \Phi_J(\Mc_{\lambda}(\Kc))$ is compact as the image of a compact by a continuous map, and by the classical universal approximation theorem for finite dimensional functions, for any $k \in \llbracket 1, r \rrbracket$, there exists a feedforward neural network $\Bc_k : \R^J \to \R$ and $\Tc_k : I \times \R^d \to \R$ such that for given $\delta > 0$
\begin{align}
\begin{cases}
        \underset{1 \leq k \leq r}{\text{ max}} \underset{z \in \tilde{\Kc}}{\text{ sup }} | f_k(z) - \Bc_k(z)| \leq \delta, \\
                \underset{1 \leq k \leq r}{\text{ max}} \underset{z \in \Kc}{\text{ sup }} | g_k(z) - \Tc_k(z)| \leq \delta,
\end{cases}
\end{align}
Denoting now the cylindrical $\text{DeepONetCyl}$ map on $I \times \R^d \times \Mc_{\lambda}(\Kc)$ as  $\text{DeepONetCyl}(\mu)(u,x) :=\sum_{k=1}^{r} \Bc_k(\Phi_J(\mu)) \Tc_k(u,x)$.
We now show that 
\begin{align}
    \underset{(u,x,\mu) \in \Kc \times \Mc_{\lambda}(\Kc)}{\text{ sup }} |V(u,x,\mu) - \text{DeepONetCyl}(u,x,\mu) | \leq \epsilon.
\end{align}
Denote the positive and  finite constants $M_f$ and $M_g$ as 
\begin{align}
\begin{cases}
    M_f := \underset{1 \leq k \leq r}{\text{ max}}  \underset{z \in \tilde{\Kc}}{\text{ sup }} |f_k(z)|, \\
    M_g := \underset{1 \leq k \leq r}{\text{ max}}  \underset{z \in \Kc}{\text{ sup }} |g_k(z)|,
\end{cases}
\end{align}
Then $\underset{z \in \tilde{\Kc}}{\text{ sup }} | \Bc_k(z)| \leq M_f + \delta $ and   $\underset{z \in \Kc}{\text{ sup }} | \Tc_k(z)| \leq M_g + \delta $. 

\noindent Let $(k,\mu) \in \Kc \times \Mc_{\lambda}(\Kc)$. Then, we have
\begin{align}
    \lVert V - \text{DeepONet} \rVert_{\infty; \Kc \times \Mc_{\lambda}(\Kc)} &= \lVert \sum_{k=1}^{r} f_k(\Phi_J(\cdot))g_k -  \sum_{k=1}^{r} \Bc_k(\Phi_J(\cdot))\Tc_k \lVert_{\infty; \Kc \times \Mc_{\lambda}(\Kc)} \\
    &\leq \sum_{k=1}^{r} \Big( \lVert f_k \rVert_{\infty, \tilde{\Kc}} \lVert g_k - \Tc_k \rVert_{\infty;\Kc}  + \lVert \Tc_k \rVert_{\infty,\Kc} \lVert  f_k - \Bc_k \rVert_{\infty;\tilde{\Kc}} \Big) \\
    &\leq  r (  M_f  \delta + (M_g+ \delta) \delta)
\end{align}
Choosing $\delta$ small enough is enough to get the required result.

\noindent     \textbf{Step n°3 : Approximation theorem on $L^2(\rho)$} 

Let $\epsilon > 0$ and $\nu$ a probability measure on $\Mc_{\lambda}$. Given $M \geq 0$ we truncate the map $V$ by defining $V_M$ on $I \times \R^d \times \Mc_{\lambda}$ as 
\begin{align}
V_M(u,x,\mu) :=
\begin{cases}
V(u,x,\mu) & \text{if $|V(u,x,\mu)| \leq M$,} \\
M \frac{V(u,x,\mu)}{|V(u,x,\mu)|} & \text{if  $|V(u,x,\mu)| > M$,}
\end{cases}
\end{align}
so that $|V_M(u,x,\mu)| \leq M$ for every $(u,x,\mu) \in I \times \R^d \times \Mc_{\lambda}$. By definition of $V_M$, we have the point-wise convergence: $V_M \underset{M \to \infty}{\to} V$. Therefore, for a fixed $\mu \in \Mc_{\lambda}$ by looking over the space $L^2(\mu)$ and noticing that $|V_M(\cdot,\cdot,\mu) - V(\cdot,\cdot,\mu)|  \leq |V(\cdot,\cdot,\mu)| \in L^2( \mu)$. Therefore, it implies by the dominated convergence theorem that the map $\mu \mapsto |V_M(\cdot,\cdot,\mu) - V(\cdot,\cdot, \mu)|_{\mu}$ converges point-wisely to $0$ and since $|V_M(\cdot,\cdot,\mu) - V(\cdot,\cdot, \mu)|_{\mu} \leq \ | V(\cdot,\cdot,\mu)|_{\mu} \in L^2(\rho)$ by assumption on $V$,  we conclude by the convergence dominated theorem that $\lVert V-V_M \rVert^2_{L^2(\rho)} \to 0$. We can therefore choose $M >\frac{\sqrt{\epsilon}}{4}$ such that 
\begin{align}\label{eq : truncated_approximation}
    \lVert V- V_M \rVert^2_{L^2(\rho)} \leq \frac{\epsilon}{8}.
\end{align}
Now, we consider some compact set $\Kc \subset I \times \R^d$ such that $\rho \big(\Mc_{\lambda} \backslash \Mc_{\lambda}(\Kc) \big) \leq \frac{\epsilon}{80 M^2}$ following Lusin's theorem on complete and separable metric spaces (see Theorem C.1 in \cite{lanthaler2022error} ) and we note that $V_M$ is continuous on $I \times \R^d \times \Mc_{\lambda}$. Applying the universal approximation theorem from \textbf{Step n°2}, we get $r,J \in \N$, $(\varphi_k)_{1 \leq k \leq r} \in \Cc(\R^J)$ and $(\psi_k)_{1 \leq k \leq r} \in \Cc(I \times \R^d)$ such that 
\begin{align}
    \underset{(u,x) \in \Kc, \mu \in \Mc_{\lambda}(\Kc)}{\text{ sup }} |V_M(u,x,\mu) - \sum_{k=1}^{r} \varphi_k(\Phi_J(\mu)) \psi_k(u,x)| \leq \frac{\sqrt{\epsilon}}{4}.
\end{align}
Therefore, we have 
\begin{align}\label{eq : ineq_proof_approximation_theorem}
    | \sum_{k=1}^{r} \varphi_k(\Phi_J(\mu)) \psi_k(u,x)| &\leq |V_M(u,x,\mu)| + |V_M(u,x,\mu) -  \sum_{k=1}^{r} \varphi_k(\Phi_J(\mu)) \psi_k(u,x) |, \\
    &\leq  M + \frac{\sqrt{\epsilon}}{4} \leq 2M, \quad (u,x) \in \Kc, \mu \in \Mc_{\lambda}(\Kc).
\end{align}
We note that after suitably modifying the linear output layers of the branch and trunk nets neural nets of the previous \text{DeepONetCyl} (see Theorem 3.1 in \cite{lanthaler2022error}), we can rewrite 
\begin{align}\label{eq : transformation_deepOnet}
     \sum_{k=1}^{r} \varphi_k(\Phi_J(\mu)) \psi_k(u,x)=\sum_{k=1}^{l} \Bc_k(\Phi_J(\mu)) \Tc_k(u,x),
\end{align}
where the trunk neural nets $(\Tc_k)_{1 \leq k \leq l}$ are orthonormal in $L^2(\mu)$ for some $l \leq r$. In particular, we have
\begin{align}\label{eq : L2norm_comput}
    |\sum_{k=1}^{l} \Bc_k(\Phi_J(\mu)) \Tc_k(\cdot,\cdot)|^2_{\mu} = | \Bc(\Phi_J(\mu))|^2, \quad \mu \in \Mc_{\lambda},
\end{align}
and
\begin{align}
    | \Bc(\Phi_J(\mu))| \leq 2M, \quad \forall \mu \in \Mc_{\lambda},
\end{align}
where we denoted $\Bc(\mu) := (\Bc_1(\mu), \ldots,\Bc_l(\mu))$.
Now, by the clipping lemma (see Lemma C.2 in \cite{lanthaler2022error}), there exists a neural network $\gamma : \R^l \to \R^l$, satisfying
\begin{align}\label{eq : clipping_lemma}
\begin{cases}
    |\gamma(y) - y | &\leq \frac{\sqrt{\epsilon}}{4}, \quad \text{ if $|y| \leq M + \frac{\sqrt{\epsilon}}{4}$},\\
    |\gamma(y)| &\leq  2M, \quad \forall y \in \R^l.
\end{cases}
\end{align}
We now define the following \text{DeepONetCyl} over $I \times \R^d \times \Mc_{\lambda}$ as 
\begin{align}\label{eq : def_final_DeepOnetCyl}
    \text{DeepONetCyl}(\mu)(u,x) := \sum_{k=1}^{l} \gamma(\Bc(\Phi_j(\mu)) \Tc_k(u,x),
\end{align}
and we have 
\begin{align}
    | \text{DeepONetCyl}(\mu)(\cdot,\cdot)|_{\mu} = |\gamma(\Bc(\Phi_J(\mu))| \leq  2M,
\end{align}
by \eqref{eq : clipping_lemma} and since $(\Tc_{k})_{1 \leq k \leq l}$ are orthonormal in $L^2(\mu)$. Moreover, we have
\begin{align}
    | V_M(\cdot,\cdot,\mu)- \text{DeepONetCyl}(\mu)(\cdot,\cdot) |_{\mu} &\leq | V_M(\cdot,\cdot,\mu)-   \sum_{k=1}^{l} \Bc_k(\Phi_J(\mu)) \Tc_k(\cdot,\cdot) |_{\mu}  \\
    &\quad+ |    \sum_{k=1}^{l} \Bc_k(\Phi_J(\mu)) \Tc_k(\cdot,\cdot) - \text{DeepONetCyl}(\mu)(\cdot,\cdot) |_{\mu} \\
    &\leq \frac{\sqrt{\epsilon}}{4} + \frac{\sqrt{\epsilon}}{4} = \frac{\sqrt{\epsilon}}{2}.
\end{align}
from \eqref{eq : ineq_proof_approximation_theorem}, \eqref{eq : L2norm_comput} and \eqref{eq : clipping_lemma}.
We therefore have 
\begin{align}
    C &:= \int_{\Mc_{\lambda}} |V_M(\cdot,\cdot,\mu) - \text{DeepONetCyl}(\mu)(\cdot,\cdot)|^2_{\mu} \rho(\d \mu) \\
    &\leq \int_{\Mc_{\lambda}(\Kc)} |V_M(\cdot,\cdot,\mu) - \text{DeepONetCyl}(\mu)(\cdot,\cdot)|^2_{\mu} \rho(\d \mu)   \\
    & \qquad + \;   2 \int_{\Mc_{\lambda}\backslash \Mc_{\lambda}(\Kc)} \big( | V_M(\cdot,\cdot,\mu)|^2_{\mu}  | +\text{DeepONetCyl}(\mu)(\cdot,\cdot)|^2_{\mu} \big) \rho(\d \mu) \\
    &\leq  \frac{\epsilon}{4} +   2 \big( M^2 + 4M^2 \big) \frac{\epsilon}{80 M^2} = \frac{3 \epsilon}{8},
\end{align}
where we used $|V_M| \leq M$ and $|\text{DeepOnetCyl}(\mu)(\cdot,\cdot)|_{\mu} \leq 2M$. Now, recalling \eqref{eq : truncated_approximation}, we have
\begin{align}
     D &:=\int_{\Mc_{\lambda}} | V(\cdot,\cdot,\mu) - \text{DeepONetCyl}(\cdot,\cdot,\mu)|^2 \rho(\d \mu)
     \\
     & \leq2   \int_{\Mc_{\lambda}} | V(\cdot,\cdot,\mu) - V_M(\cdot,\cdot,\mu)|^2_{\mu} \rho(\d \mu) \\
     &\quad+ 2 \int_{\Mc_{\lambda}} |V_M(\cdot,\cdot,\mu) - \text{DeepONetCyl}(\cdot,\cdot,\mu)|^2_{\mu} \rho(\d \mu) \leq \frac{2 \epsilon}{8} + 2 \frac{3 \epsilon}{8} = \epsilon.
\end{align}
The proof is therefore completed.

\subsection{Data generation}\label{subsec : dataa_generation}

The training of neural networks for approximating mean field functions relies on sampling $\mu \in \Mc_{\lambda}$ and a random variable $(U,X)$ whose law  is $\mu$. We first  give a Lemma which provides a simple way to sample on $\Mc_{\lambda}$.
\begin{Lemma}\textnormal{(Sampling in $\Mc_{\lambda}).$}\label{lemma : samplingPlambda}

\noindent Let $\nu$ be a non-atomic reference probability  in $\Pc_2(\R^d)$ and $(\mu^u)_u \in L^2(I ;\Pc_2(\R^d))$ Then, there exists a measurable map $T \in L^2(\lambda \otimes \nu; \R^d)$ such that $T(u,\cdot)\sharp \nu = \mu^u$ $\lambda(\d u)-$a.e.
\end{Lemma}
\begin{proof}
\noindent      The proof is a simple application of the randomization lemma (see \cite{kallenberg2002foundations}). Indeed, since the map $I \ni u \mapsto \mu^u \in \Pc_2(\R^d)$ is measurable, there exists a measurable map $F : I \times [0,1] \to \R^d$ such that $\big(U,F(U,R)\big) \sim \mu(\d u , \d x) = \mu^u(\d x) \lambda(\d u ) $ where $(U,R) \sim \Uc([0,1]) \otimes \Uc([0,1])$ where $R$ is used for randomization and $U$ for the labeling of the agents. Disintegrating $\mu$ over $I$, we get that $F(u,\cdot) \sharp \Uc([0,1]) = \mu^u \text{ $\lambda(\d u)$-a.e}$.

     Let $\Phi : \R^d \to [0,1]$ be a measurable  isomorphism satisfying $\Phi \sharp \nu = \Uc([0,1])$ (which exists since $\nu$ is non-atomic on a Borel space). Defining the measurable map $T$ on $I \times \R^d$ into $\R^d$ as $T(u,x) := F(u, \Phi(x) \big)$, we end up with the result. We now verify that $T \in L^2(\lambda \otimes \nu ; \R^d)$, i.e.
     \begin{align}\label{eq : T_in_L2}
         \int_{I \times \R^d } | T(u,x)|^2 \nu(\d x) \lambda(\d u ) < + \infty.
     \end{align}
     Now, \eqref{eq : T_in_L2} follows noticing that for $\lambda(\d u )-\text{a.e}$, $\int_{\R^d} |T(u,x)|^2 \nu(\d x) =\int_{\R^d} |y|^2 \mu^u(\d y)$ since $T(u,\cdot) \sharp \nu = \mu^u $ and because $(\mu^u)_u \in L^2(I;\Pc_2(\R^d))$.
\end{proof}
From Lemma \ref{lemma : samplingPlambda}, we design two ways to sample on $\Mc_{\lambda}$.

\begin{enumerate}
    \item [] \textbf{S1:} Given a non-atomic measure $\nu \in \Pc_2(\R^d)$, we sample a dense class of maps $(T_{\theta})_{\theta \in \Theta} \in L^2(\lambda \otimes \nu ; \R^d)$ (for instance $(T_{\theta})_{\theta \in \Theta}$ can be sampled as a parametrized class of neural networks for a 
    set $\Theta \subset \R^m$.  
    In this case, $\P_{(U,T_{\theta}(U,Y))} \in \Mc_{\lambda}$, where $U$ and $Y$ are random variables on $(\Omega,\Fc,\P)$ s.t $(U,Y) \sim \lambda \otimes \nu$.
    \item [] \textbf{S2:} Alternatively, we can sample a class of non-atomic measures $\nu  \in \Sc \subset \Pc_2(\R^d)$, given a fixed map $T \in L^2(\mu \otimes \nu)$. In this case,  we start by sampling  $(U,Y) \sim \lambda \otimes \nu$ where $\nu \in S$ and we set  $(U,T(U,Y))$ whose law belongs to $\Mc_{\lambda}$ by construction. However,  this method only helps us to sample in the set $\big \lbrace \lambda(\d u)  T(u,\cdot){\sharp} \nu(\d y) : \nu \in \Sc \big \rbrace   $.
\end{enumerate}

\begin{Remark}\textnormal{(Special case of $d=1$).}

\noindent 

\normalfont

\noindent In the special case of $d=1$, the map $T$ can be explicitly constructed. Indeed, we take $\nu = \Uc([0,1])$. Then, given $(\mu^u)_u \in \Pc_2(\R) $ we denote  its cumulative distribution function  as $F_{\mu^u}(x) := \mu^u( ( - \infty;x])$ and we define its quantile function $Q_{\mu^u}$  as 
\begin{align}
   Q_{\mu^u}(t):= \text{inf } \lbrace x \in \R : F_{\mu^u}(x) \geq t \big \rbrace, \quad t \in (0,1).
\end{align}
Then, defining $T(u,x) :=Q_{\mu^u}(x)$ gives the required result.

\end{Remark}

\subsection{Training the mean-field operator}

The algorithm consists in the following parts.

\begin{enumerate}
    \item In order to use the $\text{DeepONetCyl}$ structure and the universal approximation of Theorem \ref{thm : universal_approximation_theorem}, we need to consider a finite but dense functions in order to compute the quantities $\Phi_J(\mu)$. In practice, we samples polynomial maps and we choose $r$ and $J$ as hyperparameters. In fact, for $1 \leq i \leq J$, we choose the moment maps $\varphi_i(u,x) := |x|^i + u^i$ and we notice that $\langle \varphi_i, \mu \rangle =   \E_{(U,X) \sim \mu} [ |X|^i ]$ up to a constant for $\mu \in \Mc_{\lambda}$.
    \item We draw samples $(\mu^{(m)} )_{1 \leq m \leq M}$ of probability measures on $\Mc_{\lambda}$ (recalling  the previous section).
    We introduce two neural networks $\Tc^{\theta_1}$ parametrized with $\theta_1$ and  $ \Bc^{\theta_2}$ parametrized by $\theta_2$. We then minimize over the parameters $ \theta := (\theta_1,\theta_2)$ the following loss function
    \begin{align}
        L(\theta) := \frac{1}{M} \sum_{m=1}^{M} \E_{(U,X) \sim \mu^{(m)}} \bigg[ \Big| V(U,X, \mu^{(m)}) - \sum_{k=1}^{r} \Tc_{k}^{\theta_1}\big(U,X) \Bc_{k}^{\theta_2}(\Phi_J(\mu^{(m)})  \big )    \Big|^2 \bigg].
    \end{align}
\end{enumerate}
\vspace{0.2 cm}
It can be summarized in the following algorithm :\\
\begin{algorithm2e}[H]
\DontPrintSemicolon 
\SetAlgoLined 
\footnotesize
{\bf Input}: 
Number of distributions $M$, batch size $N$, number of  epoch $e$, learning rate $\rho$, $\theta=(\theta_1,\theta_2)$ initial parameters of neural networks  $\Tc^{\theta_1} =(\Tc^{\theta_1}_1 , \ldots, \Tc^{\theta_1}_r)$ and  $ \Bc^{\theta_2} =(\Bc^{\theta_2}_1,\ldots, \Bc^{\theta_2}_r)$, $J$ number of moments, $r$ number of sensors.

\For{each epoch $e$}{
    Sample $(U_{m,n}, X_{m,n}) \sim \mu^{(m)}$ for $n \in \llbracket 1,N \rrbracket$ where  $\mu^{(m)} \in \Mc_{\lambda} $ for $m \in \llbracket 1, M \rrbracket$\;
    Calculate $\Phi_{m} = (\frac{1}{N} \sum_{n=1}^N  |X_{m,n}|^j)_{j \in \llbracket 1, J \rrbracket} $ for $m \in \llbracket 1,M \rrbracket$ \; 
    Calculate target function $$J(\theta)= \frac{1}{MN} \sum_{n=1}^M \sum_{j=1}^N \Big(V(U_{m,n},X_{m,n}, \mu^{(m)}) -  \sum_{k=1}^r \Tc^{\theta_1}_{k}(U_{m,n} , X_{m,n}) \Bc^{\theta_2}_k (\Phi_m)\Big)^2$$ \;
    $\theta = \theta - \rho \nabla J(\theta)$
 }
{\bf Return:} $\theta$
\caption{Non exchangeable mean field network approximation of a map }\label{algo : echangeable_approx}
\end{algorithm2e}

\section{Numerical experiments} \label{sec:num}

We test our algorithms by computing the mean-squared error (MSE) for different cases of non exchangeable mean-field functions $V$ on $I \times \R \times \Mc_{\lambda}$. As in the literature, we use a graphon map $G$ (i.e. a measurable map on $I \times I$ into $\R_+$) to represent the interaction between two agents. We propose the two following maps to illustrate our Algorithm  \ref{algo : echangeable_approx}.
\begin{enumerate}
    \item [] \textbf{V1}:   A first order non exchangeable mean-field interaction 
    \begin{align}
        V(u,x,\mu) :=  x - \E_{(U,X) \sim \mu} \big[ G(u,U) X \big].
    \end{align}
    \item [] \textbf{V2}: A second order non exchangeable mean-field interaction
    \begin{align}
        V(u,x,\mu) := \E_{(U,X) \sim \mu} \big[ (x- G(u,U) X)^2 \big].
    \end{align}
\end{enumerate}
We first generate samples in $\Mc_{\lambda}$ by using $\textbf{S2}$ in Section \ref{subsec : dataa_generation}, i.e., we fix a transport map $T$ and we sample a class of non-atomic measures.
We first propose two simple cases for choice of map $T$ given the unknown link between $X$ and $U \sim \Uc([0,1])$.

\begin{enumerate}
    \item [] \textbf{T1}: $T(u,y) := uy$, 
    \hspace{3cm} 
    \textbf{T2}: $T(u,y) :=  uy + (uy)^2.$
\end{enumerate}
 We then draw $X=T(U,Y)$, with $Y$  sampled according to a distribution picked randomly as a mixture of five gaussian laws.
In fact, for each $m$, we sample for $k \in \llbracket 1, 5 \rrbracket$,
$W_{m,k} \sim \Uc([0,1])$, $(\mu_{m,k}, \sigma_{m,k})  \sim \Uc([0,1])^2$, and then we compute for any $n \in \llbracket 1, N \rrbracket$
\begin{align}
    Y_{m,n} =   \frac{1}{\sum_{k=1}^5 W_{m,k}}\sum_{k=1}^5 W_{m,k} Z_{m,k,n}, \quad   \text{ with} \quad Z_{m,k,n} \sim \mathcal{N}(\mu_{m,k},\sigma_{m,k}^2). 
\end{align}
 \\
We consider two cases of interacting functions:
\begin{enumerate}
    \item   A smooth graphon 
    \begin{align}
        G_1(u,v):= e^{-uv}.
    \end{align}
    \item  A block-wise graphon
    \begin{align}\label{eq : G2_map}
        G_2(u,v) := L \sum_{i=1}^L \frac{1}{1 + e^{((i-1)L(u-v))}} 1_{(\frac{i-1}{L},\frac{i}{L}]}(u) 1_{(\frac{i-1}{L},\frac{i}{L}]}(v),
    \end{align} 
    as in \cite{alvarez2025contracting}, modeling $L=5$ teams, each team with a given interaction between agents (in the first team  agents have an homogeneous interaction) and the different teams have no interaction. The function is represented on Figure \ref{fig:G2}.
    \begin{figure}[h!]
        \centering
        \includegraphics[width=0.5\linewidth]{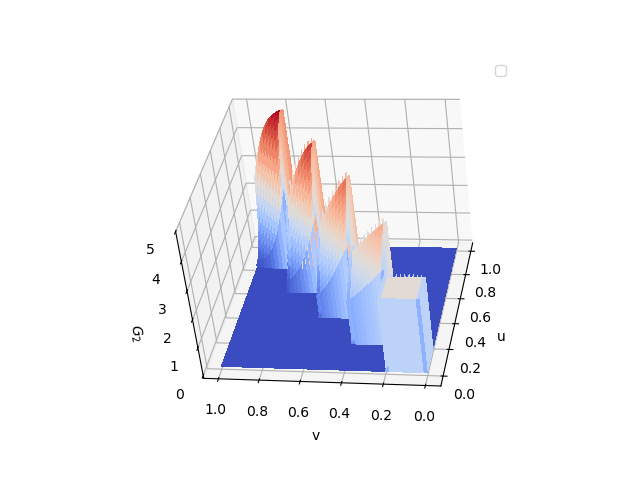}
        \caption{$G_2$ graphon.}
        \label{fig:G2}
    \end{figure}
\end{enumerate}
All the tests are achieved using the ADAM optimization method with a learning rate of $0.001$. The default value for $r$ is $10$ and we consider two architectures for the networks $\Tc^{\theta_1}$ and  $ \Bc^{\theta_2}$ :
\begin{itemize}
    \item The first one is the classical feedforward network using a $\tanh$ activation function using 3 hidden layers of 10 neurons.
    \item The second one is a Kolmogorov  Arnold Network  \cite{liu2024kan} using two hidden layers of 10 neurons and a grid size equal to 5. In this version, one dimensional functions are approximated using splines  (Spline KAN). Notice that  no universal approximation theorem is available for this network.
    \item The third one is  a Kolmogorov Arnold Network  using P1 type finite element functions with adapting support to approximate one dimensional functions  (P1KAN) \cite{warin2024p1}. This network is specially effective to approximate irregular functions. Its convergence is supported by a universal approximation theorem. With this network, we use two hidden layers of 10 neurons and a grid size equal to 10. 
\end{itemize}
In the algorithm we use $N=50000$ samples to approximate the distribution and take one distribution at each iteration of the gradient algorithm.

 We now illustrate the convergence of the non exchangeable network using a GPU Nvidia H100 94Go HBM2.  All convergence plots are  given with 100000 iterations calculating the accuracy every hundred of iterations and smoothing the result obtained with a rolling window of 10 values.

On Figures \ref{fig:F1G1}, \ref{fig:F2G1} we give the convergence depending on the number of moments. Not surprisingly with mean-field function in \textbf{V1}, the convergence rate is independent of the number of moment,  while with mean-field function in \textbf{V2} at least two moments are necessary.

 \begin{figure}[H]
     \centering
     \begin{minipage}[t]{0.47\linewidth}
    \includegraphics[width=\linewidth]{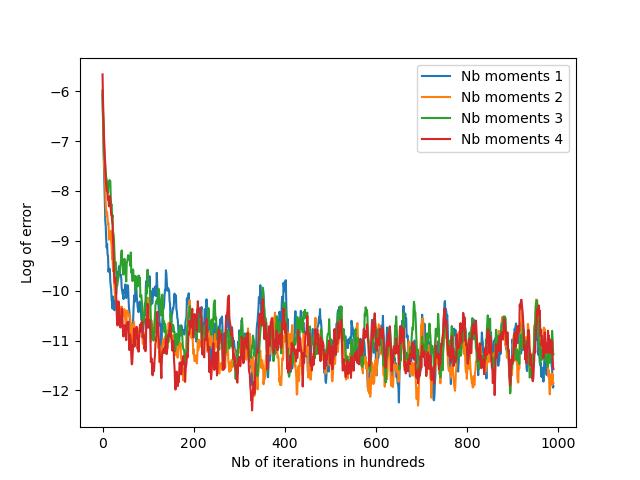}
    \caption*{ \textbf{T1}}
    \label{fig:enter-label}
    \end{minipage}
     \begin{minipage}[t]{0.47\linewidth}
    \includegraphics[width=\linewidth]{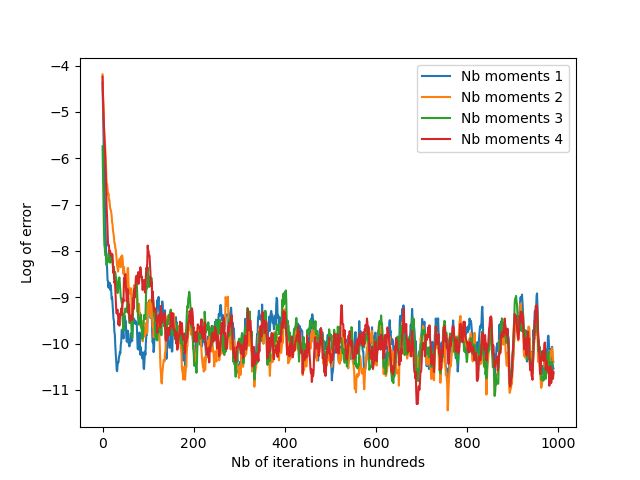}
    \caption*{ \textbf{T2}}
    \label{fig:enter-label}
    \end{minipage}
     \caption{Mean-field function in \textbf{V1}, graphon  $G_1$  using  the feedforward network :  \text{log} of the error obtained during training depending on $J$ the number of moments.}
     \label{fig:F1G1}
 \end{figure}

 \begin{figure}[H]
     \centering
     \begin{minipage}[t]{0.47\linewidth}
    \includegraphics[width=\linewidth]{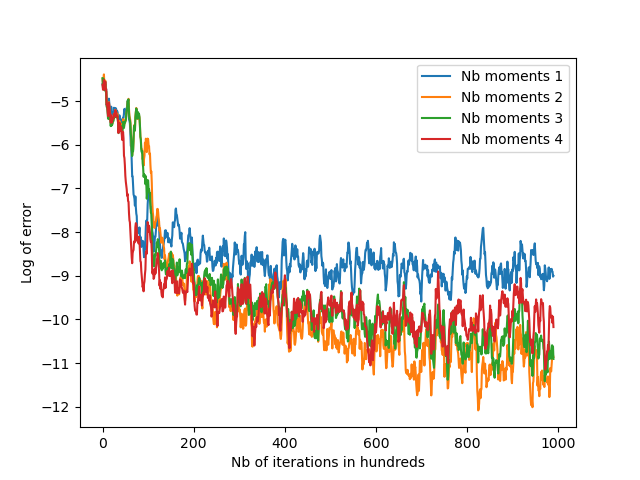}
    \caption*{ \textbf{T1}}
    \label{fig:enter-label}
    \end{minipage}
     \begin{minipage}[t]{0.47\linewidth}
    \includegraphics[width=\linewidth]{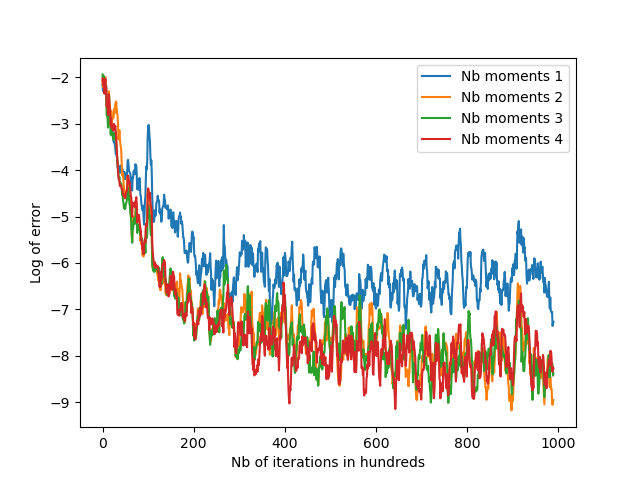}
    \caption*{ \textbf{T2}}
    \label{fig:enter-label}
    \end{minipage}
     \caption{Mean-field function in \textbf{V2}, graphon $G_1$ using the feedforward network :  $\log$ of the error obtained during training on $J$ the number of moments.}
     \label{fig:F2G1}
 \end{figure}
 
 Figure \ref{fig:G1SplineVFeed} illustrates the fact that choice of the network architecture is crucial to get a very good convergence. For these very regular cases, Spline KAN outperforms the two other networks that behave similarly.
 KAN networks are known to be more expensive to use than feedforwards but here computing times are very similar. For mean-field function in \textbf{V2}, with transport map from \textbf{T2},  with one moment, 100 iterations takes 3.56 seconds with a feedforward while it takes 3.74 seconds with the Spline KAN. Using 4 moments, the feedforward network takes 3.61 seconds while the Spline KAN takes 3.75 seconds. Results are similar with the P1KAN network. The crucial point here is to effectively calculate the interaction matrix.
 
 \begin{figure}[H]
     \centering
     \begin{minipage}[t]{0.47\linewidth}
    \includegraphics[width=\linewidth]{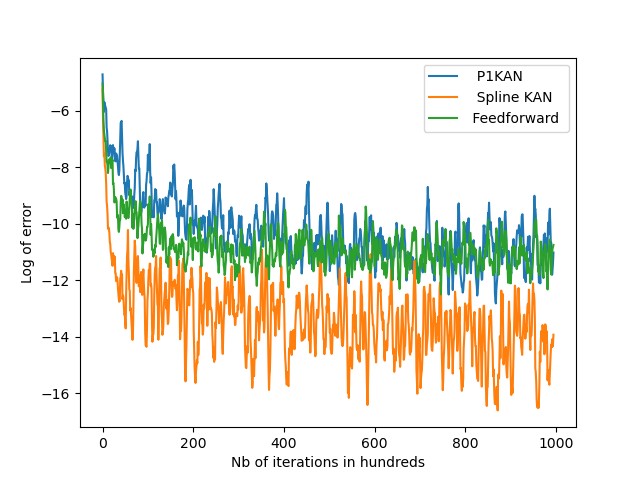}
    \caption*{  \textbf{V1-T1} }
    \label{fig:enter-label}
    \end{minipage}
     \begin{minipage}[t]{0.47\linewidth}
    \includegraphics[width=\linewidth]{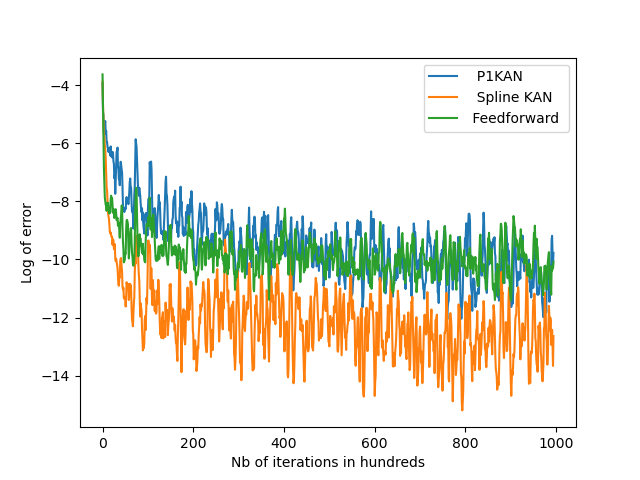}
    \caption*{ \textbf{V1-T2}}
    \label{fig:enter-label}
    \end{minipage}
     \begin{minipage}[t]{0.47\linewidth}
    \includegraphics[width=\linewidth]{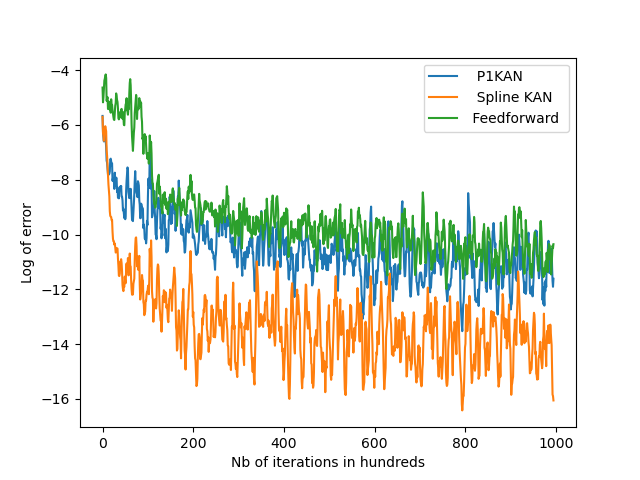}
    \caption*{  \textbf{V2-T1}}
    \label{fig:enter-label}
    \end{minipage}
     \begin{minipage}[t]{0.47\linewidth}
    \includegraphics[width=\linewidth]{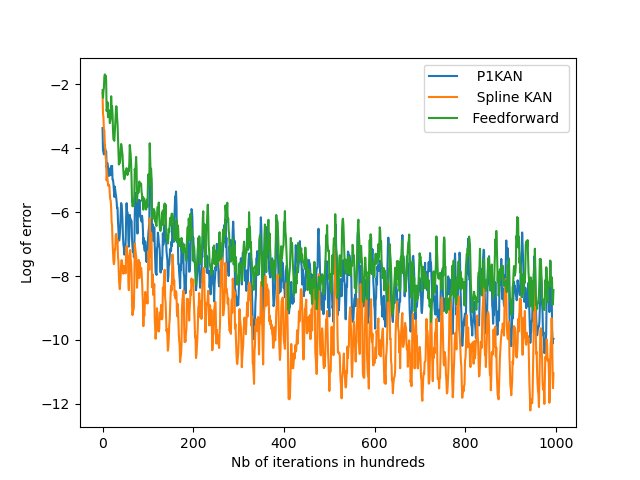}
    \caption*{ \textbf{V2-T2}}
    \label{fig:enter-label}
    \end{minipage}
     \caption{ Comparison of feedforward and KAN  with $G_1$ interaction function using 3 moments}
     \label{fig:G1SplineVFeed}
 \end{figure}

On Figure \ref{fig:F2G2}, we plot the convergence using the $G_2$ function which is far more irregular for the two cases of \textbf{T2}. The  convergence is much more difficult to achieve and much more erratic.
 \begin{figure}[H]
     \centering
     \begin{minipage}[t]{0.49\linewidth}
    \includegraphics[width=\linewidth]{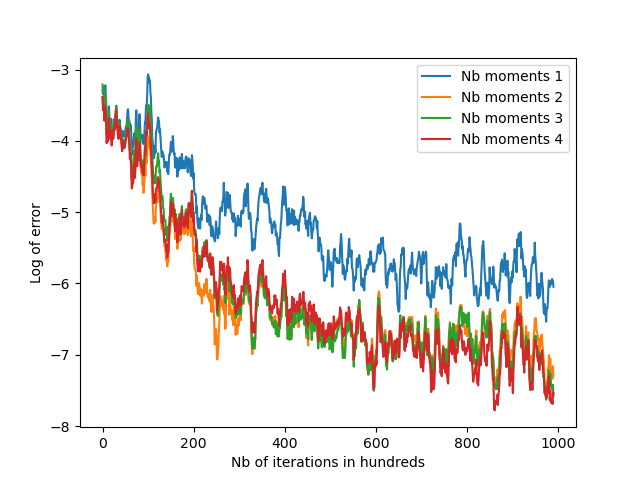}
    \caption*{ \textbf{T1}}
    \label{fig:enter-label}
    \end{minipage}
     \begin{minipage}[t]{0.49\linewidth}
    \includegraphics[width=\linewidth]{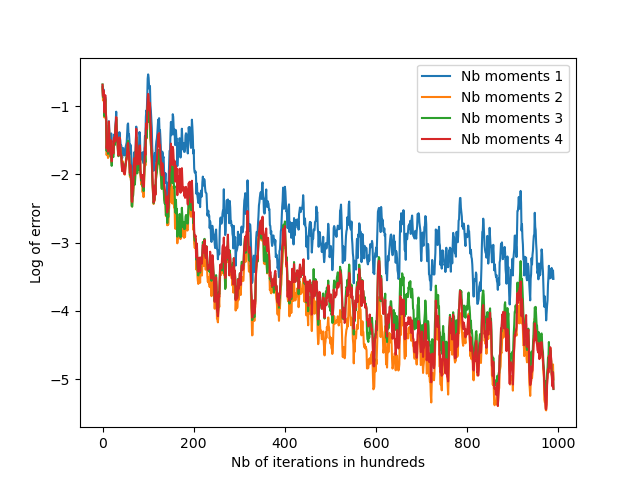}
    \caption*{ \textbf{T2}}
    \label{fig:enter-label}
    \end{minipage}
     \caption{Mean-field function from \textbf{V2}, graphon $G_2$ interaction  using the feedforward network :  $\log$ of the error obtained during training on $J$ the number of moments}
     \label{fig:F2G2}
 \end{figure}
 On Figure \ref{fig:G2SplineVFeed}, we show that the KAN networks converge better and faster than the feedforward.  As the functions to approximate are rather irregular, the P1KAN network outperforms the two other ones.
 \begin{figure}[H]
     \centering
     \begin{minipage}[t]{0.49\linewidth}
    \includegraphics[width=\linewidth]{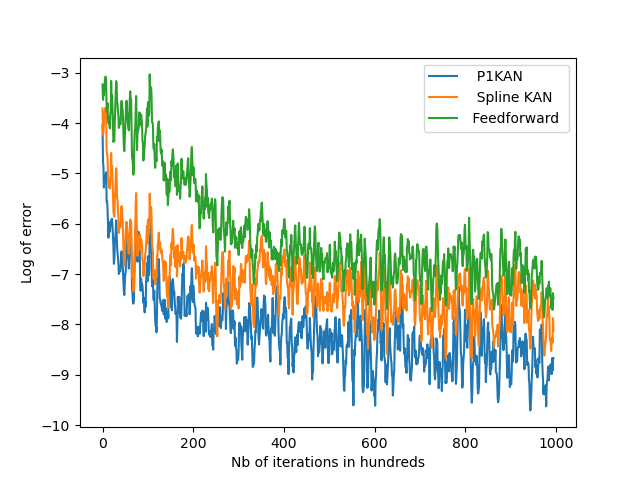}
    \caption*{  \textbf{T1}}
    \label{fig:enter-label}
    \end{minipage}
     \begin{minipage}[t]{0.49\linewidth}
    \includegraphics[width=\linewidth]{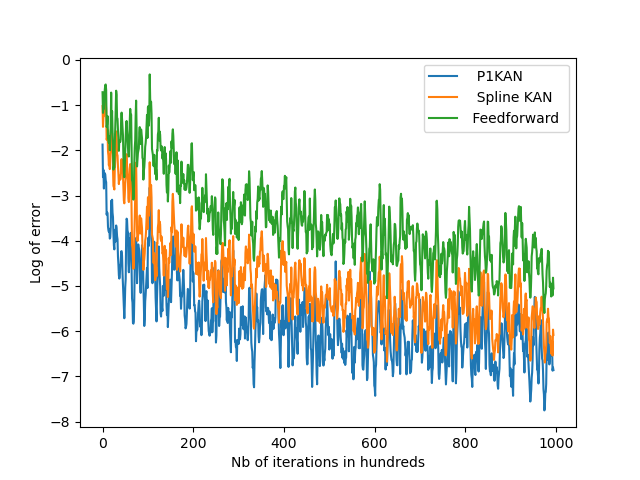}
    \caption*{ \textbf{T2}}
    \label{fig:enter-label}
    \end{minipage}
     \caption{ Comparison of feedforward and KAN  with $G_2$ interaction function using 3 moments for mean-field function \textbf{V2}.}
     \label{fig:G2SplineVFeed}
 \end{figure}
 
On Figure \ref{fig:sensor},  we show that the number of sensors $r$ used is not critical.
 \begin{figure}[H]
     \centering
    \begin{minipage}[t]{0.49\linewidth}
     \includegraphics[width=\linewidth]{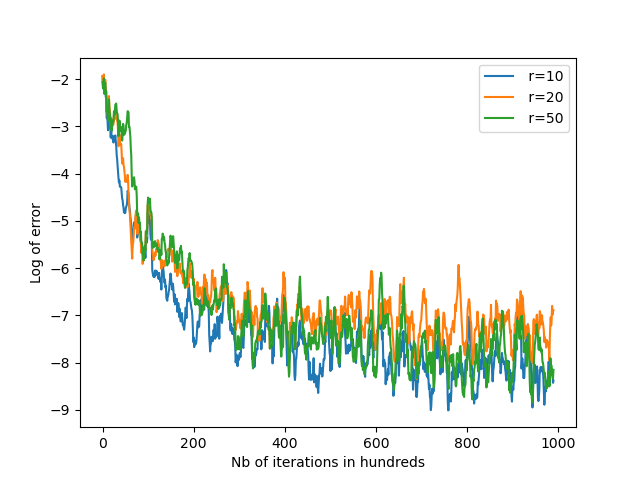}
     \caption*{$G_1$ interaction function.}
     \end{minipage}
         \begin{minipage}[t]{0.49\linewidth}
     \includegraphics[width=\linewidth]{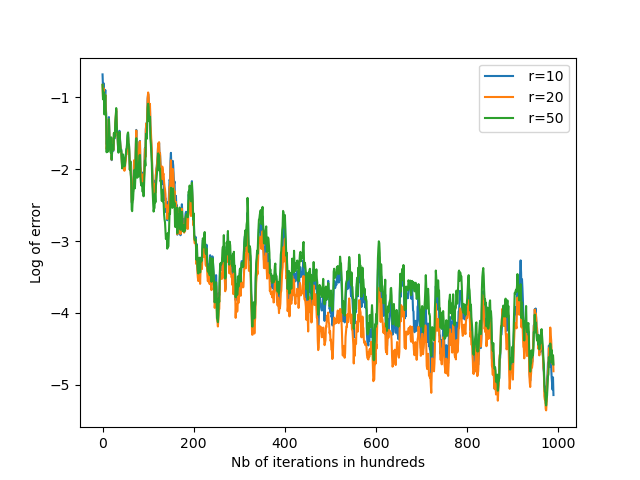}
     \caption*{$G_2$ interaction function.}
     \end{minipage}
     \caption{Impact of the number of sensors $r$ on the convergence for mean-field function  \textbf{V2} and transport map \textbf{T2} with 3 moments.}
     \label{fig:sensor}
     \end{figure}

 Finally, we adopt the sampling method in \textbf{S1}. We suppose now that $Y \sim \Nc(0,1)$, and we consider the family of random transport maps $\omega \mapsto T(\omega,\cdot,\cdot) \in L^2(\lambda \otimes \Nc(0,1))$:
 \begin{align}
     T(\omega, u,y) := A(\omega) uy + B(\omega) ( uy + u^2 y^2),
 \end{align}
 where $(A,B) \sim \Uc([0,1])^{\otimes^2}$ and then we sample $X=T(U,Y)$ with $(U,Y)$ independent of $(A,B)$.
 
 On Figure \ref{fig:SplineVFeedMixOp}, we show with $J=3$ moments that the convergence is achieved especially when a KAN network is used. We observe that the convergence curve is smoother than in the sampling method \textbf{S2}.
 Again with the $G_2$ interaction function, the P1KAN network outperform the two other networks.
 \begin{figure}[H]
     \centering
     \begin{minipage}[t]{0.49\linewidth}
    \includegraphics[width=\linewidth]{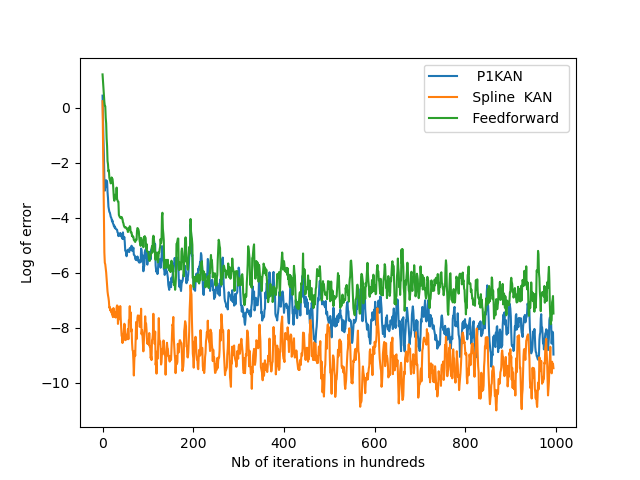}
    \caption*{  Graphon $G_1$}
    \label{fig:enter-label}
    \end{minipage}
     \begin{minipage}[t]{0.49\linewidth}
    \includegraphics[width=\linewidth]{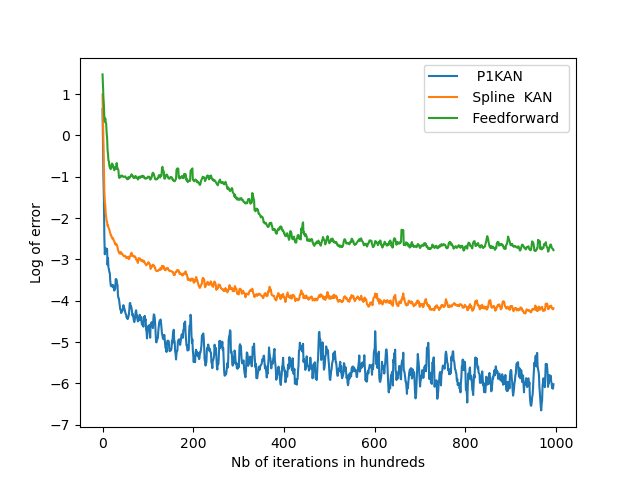}
    \caption*{Graphon $G_2$}
    \label{fig:enter-label}
    \end{minipage}
     \caption{ Comparison of feedforward and KAN  using 3 moments for mean-field function \textbf{V2} with the sampling method \textbf{S1}.}
     \label{fig:SplineVFeedMixOp}
 \end{figure}

\section{Application to optimal control of non exchangeable mean field systems}\label{sec : optimal_control} \label{sec:appli} 

In this section, we propose an application of the theory developed above to solve optimal control problem for non exchangeable mean field systems. 
Concisely, it can be formulated as the extension of the usual standard McKean-Vlasov control problem without assuming homogeneity between the agents. This naturally leads to an infinite dimensional control problem since the agents are no longer homogeneous. We present below some standard methods to solve optimal controls involving this class of controlled systems.
\subsection{Background on controlled non exchangeable mean field systems}

We first introduce the framework of optimal control on non exchangeable mean field systems in a label-state formulation. On a complete filtered probability space $(\Omega,\Fc,\P)$ satisfying the usual hypothesis, we are given the following random variables

\begin{enumerate}
    \item [$\bullet$] A uniform random variable $U$ over $[0,1]$ used for encoding the heterogeneity.
    \item [$\bullet$] A $\R^n$-valued Brownian motion $W :=(W_t)_{0 \leq t \leq T}$ independent of $U$.
    \item [$\bullet$] A $\R^d$-valued initial condition $\xi$ with law $m$ independent of $W$.
    \item [$\bullet$]  We denote by $\F^W = (\Fc_t^W)_{0 \leq t \leq T}$ the natural filtration generated by $W$ and by $\F = (\Fc_t)_{0 \leq t \leq T}$ the filtration given by $\Fc_t := \Fc_t^W \vee \sigma(U)$ where $\sigma$ denotes the $\sigma-$algebra generated by $U$, augmented with the $\P$-null sets.
\end{enumerate}

We fix drift and diffusion functions  $b, \sigma : I \times \R^d \times \Mc_{\lambda} \times A \to \R^d , \R^{d \times n}$ on which we make the following standard assumptions

\begin{Assumption}\label{assumption : regularity_coefficients}
The functions $b, \sigma : I \times \R^d \times \Mc_{\lambda} \times A \to \R^d, \R^{d \times n}$ are Borel measurable. Moreover, there exists positive constants $L \geq 0$, $M \geq 0$ such that
\begin{equation*}
\begin{cases}
| b(u,x,\mu,a) - b(u,x',\mu', a) | & \leq L \;  \big ( | x-x'| + \Wc_2(\mu,\mu') \big) \\
| \sigma(u,x,\mu,a) - \sigma(u,x',\mu',a) | &  \leq \;  L \big ( | x-x'| + \Wc_2(\mu,\mu') \big)
\end{cases}
\end{equation*}
and 
\begin{align*}
| b(u,0, \lambda \otimes \delta_{0}, a)|  + |\sigma(u,0, \lambda \otimes \delta_{0}, a) | 
&\leq \;   M \big( 1+ |a| \big),     
\end{align*}
for every $u \in I$, $x,x' \in \R^d$, $\mu,\mu' \in \Mc_\lambda$ and $a \in A$.    

\end{Assumption}
We denote by $\Ac$ the set of admissible controls that are $\F$-progressively measurable  process  valued in a convex measurable space $(A,\Ac)$ such that  $\E \Big[\int_{0}^{T} | \alpha_t|^2 \d t \Big] < \infty$, and by $\Ic_t$ the set of admissible initial conditions defined as $    \Ic_t = \big \lbrace \xi : \xi \text{ is $\Fc_t$-measurable and $\E [ | \xi|^2 \big] <  \infty$} \big \rbrace$. 

Given $\alpha \in \Ac$ and $\xi \in \Ic_t$, we consider the following controlled state process $X=(X^{\xi,\alpha}_s)_{t \leq s \leq T}$ satisfying the following SDE
\begin{align}\label{eq : SDE_label_state}
\begin{cases}
      \d X_s &= b(U,X_s, \P_{(U,X_s)}, \alpha_s) \d s + \sigma(U,X_s, \P_{(U,X_s)}, \alpha_s) \d W_s, \\
   X_t &= \xi, \\
\end{cases}
\end{align}

\begin{Theorem}\label{thm : existence_unicity_process}
    Given $\xi \in \Ic_t$ and $\alpha \in \Ac$ and under Assumption \ref{assumption : regularity_coefficients}, there exists a unique strong solution $X=(X_s)_{t \leq s \leq T}$ to Equation \eqref{eq : SDE_label_state}.
\end{Theorem}
\begin{proof}
The proof is postponed to Appendix \ref{subsec : prop_existenceunicitySDE}.   
\end{proof}

From standard estimations, there exists a positive constant $C$ such that 
\begin{align}
    \E \Big[ \underset{t \leq s \leq T}{\text{ sup }} |X_s|^2 \Big] \leq  C \Big( 1 + \E \big[ | \xi|^2 \big] +  \E \big[ \int_{t}^{T} | \alpha_s|^2 \d s  \big] \Big).
\end{align}

We next introduce the two reward functions $f : I \times \R^d \times \Mc_{\lambda} \times A \to \R$ and $g : I \times \R^d \times \Mc_{\lambda} \to \R$ on which we make the following assumptions.

\begin{Assumption}\label{assumption : cost_functional_coefficients}
The functions $f$ and $g$ are Borel measurable and there exists a constant $M \geq 0$ such that
\begin{equation*}
\begin{cases}
 - M \big( 1+ |x|^2 + \Wc_2(\mu, \lambda \otimes \delta_{0})^2 \big) \; \leq \;   f(u,x,\mu,a)  \; \leq \;  M \big( 1 +|x|^2 + \Wc_2(\mu, \lambda \otimes \delta_{0})^2 + |a|^2 \big) \\
 |g(u,x,\mu)| \;  \leq \;   M \big( 1+ |x|^2 + \Wc_2(\mu, \lambda \otimes \delta_{0})^2 \big),
\end{cases}    
\end{equation*}
for every $u \in I$, $x \in \R^d$, $\mu \in \Mc_{\lambda}$ and $a \in A$.
\end{Assumption}
We define the cost functional as the $\R$-valued map $J$ as follows
\begin{align}\label{eq : cost_functional_NEMFC}
    J(t,\xi,\alpha) := \E \Big[\int_{t}^{T} f(U,X_s, \P_{(U,X_s)}, \alpha_s) \d t + g(U,X_T, \P_{(U,X_T)}) \Big],
\end{align}
where we stressed the dependence of $J$ in the initial condition $\xi$. Under Assumptions \ref{assumption : regularity_coefficients}, \ref{assumption : cost_functional_coefficients} and  from Theorem \ref{thm : existence_unicity_process}, we get that $J(\alpha)$ is well defined for any $\alpha \in \Ac$. The goal is now to study the optimal control problem consisting in minimizing the function $J$ over $\Ac$, that is,  computing 
\begin{align}\label{eq : value_control_problem}
    V_0 := \underset{\alpha \in \Ac}{\text{ inf }} J(0,\xi,\alpha),
\end{align}
and  to find an optimal control $\alpha^{\star} \in \Ac$, i.e. s.t $V_0 = J(\alpha^{\star})$.  More generally, at any time $t \in [0,T]$, we can define the cost functional to be minimized as
\begin{align}\label{eq : value_function_time_t}
    V(t,\xi) = \underset{\alpha \in \Ac}{\text{ inf }} J(t,\xi,\alpha), \quad t \in [0,T]. 
\end{align}
In the sequel, we will analyze two known methods namely, the Pontryagin's maximum principle and the Bellman equation to solve  \eqref{eq : value_control_problem}. The Pontryagin's maximum principle will lead to the study of a fully coupled forward backward stochatic differential equation (FBSDE) from which we will be able to characterize the optimality of a control $\alpha^{\star}$ whereas the dynamic programming equation will help us to characterize the Bellman function $V$ through a partial differential equation (PDE) and from a backward recursion which enables the possibility to design several  efficient algorithms (see \cite{pham2022mean}).

\begin{Remark}
\normalfont
\noindent 
    The current formulation (i.e. \eqref{eq : SDE_label_state}-\eqref{eq : cost_functional_NEMFC}) can be viewed as a label-state formulation of the non exchangeable mean field system studied in \cite{kharroubi2025stochastic,cao2025graphon} for the stochastic maximum principle formulation and in \cite{decrescenzo2024mean} for the derivation of the dynamic programming equation and the Bellman equation. The main issue in these formulations is the necessity to deal with an uncountable continuum of controlled state processes $(X^u)_{u \in I}$ for which joint measurability over space $I \times \Omega$ is not guaranteed due to a collection of i.i.d. Brownian motions $\lbrace W^u : u \in I \big \rbrace$. Moreover, this formulation is not well suited for numerical experiments as one would need  to discretize the number of processes $(X^{i,N})_{1 \leq i \leq N}$ to consider. Since we are working at the level of the marginal laws of the processes, we work under the label-state formulation which lacks of a pathwise interpretation but for which the equality of the laws is preserved (see \cite{Mekkaoui2025analysis}). In the sequel, we will define the necessary quantities to characterize an optimal control either in the stochastic maximum principle  or in the Bellman equation but since the proofs are really similar to the ones presented in \cite{kharroubi2025stochastic,decrescenzo2024mean}, we will just give the main ideas of the proof and refer to the proofs therein for further details.
\end{Remark}

\subsubsection{Stochastic maximum principle and FBSDE equation}

For simplicity, we will look at $t=0$ but the analysis could be extended to any time $t$. For the stochastic maximum principle, we define the real-valued Hamiltonian map $H$ as
\begin{align}\label{eq : hamiltonian_def}
    H(u,x,\mu,a,y,z) &=\;  b(u,x,\mu,a) \cdot y + \sigma(u,x,\mu,a) :z + f(u,x,\mu,a),
\end{align}
where $:$ refers to the inner product between two matrices,
for any $(u,x,\mu,a,y,z) \in I \times \R^d \times \Mc_{\lambda} \times A \times \R^d \times \R^{d \times n}$. This map will be used to derive necessary and sufficient conditions for the optimality of an admissible control. It will rely on the notions of derivative and convexity over the space $\Mc_{\lambda}$ introduced in Appendix \ref{sec : analysis_Plambda}. In fact, under some regularity assumptions that will be precised on the map $H$ and defining the process $\hat{\alpha}=(\hat{\alpha}_t)_{0 \leq t \leq T}$ as
\begin{align}
    \hat{\alpha}_t = \hat{a}(U,X_t,\P_{(U,X_t)},Y_t,Z_t) = \underset{a \in A}{\text{ inf }} H(U,X_t, \P_{(U,X_t)},Y_t,Z_t,a),
\end{align}
where $(X_t,Y_t,Z_t)_{0 \leq t \leq T}$ is the solution to the following system of FBSDE
\begin{align}\label{eq : FBSDE_system_weak_formulation}
\begin{cases}
      dX_t &=  b(U,X_t,\P_{(U,X_t)},\hat{\alpha}_t)    \d t +  \sigma(U,X_t,\P_{(U,X_t)},\hat{\alpha}_t)\d W_t, \\
    dY_t &= - \partial_x H(U,X_t, \P_{(U,X_t)},P_t, Z_t, \hat{\alpha}_t) \d t - \tilde{\E} \big[ \partial_{\tilde{x}} \frac{\delta}{\delta m } H(\tilde{U},\tilde{X}_t, \P_{(U,X_t)}, \tilde{Y}_t, \tilde{Z}_t,  \tilde{\hat{\alpha}}_t) (U,X_t)\big] \d t \\
    &\quad + \;  Z_t \d W_t \\
    Y_T  &=  \partial _x g(U,X_T,\P_{(U,X_T)}) + \tilde{\E} \Big[ \partial_{\tilde{x}} \frac{\delta}{\delta m } g(\tilde{U}, \tilde{X}_T , \P_{(U,X_T)})(U,X_T) \Big], 
\end{cases}
\end{align}
where $(\tilde{U},\tilde{X}_t,\tilde{Y}_t,\tilde{Z}_t,\tilde{\hat{\alpha}}_t)$ is an independant copy $(U,X_t,Y_t,Z_t,\alpha_t)$ defined on another probability space $(\tilde{\Omega},\tilde{\Fc},\tilde{\P})$, we will show that $\hat{\alpha}$ yields an optimal control, i.e. a solution to \eqref{eq : value_control_problem}.

\vspace{1mm}

We now make the necessary assumptions on the regularity on the maps involved in the control problem to derive the stochastic maximum principle.

\vspace{1mm}

\begin{Assumption}\label{assumption2: regularity_coefficients}
\noindent 
\begin{enumerate}
    \item [(1)] The maps $(b,\sigma)$ are differentiable with respect to $(x,a)$. Moreover, the maps $\partial_x(b,\sigma)$ are assumed to be uniformly bounded.
    Finally, the maps $(x,\mu,a) \mapsto \partial_x (b,\sigma,f)(u,x,\mu,a)$ and  $(x,\mu,a) \mapsto \partial_a(b,\sigma,f)$ are continuous for $\lambda-\text{a.e}$ $u \in I$.
    \item [(2)] The maps $(b,\sigma)$ are assumed to have Fréchet differentiable linear functional derivatives $\partial_{\tilde{x}} \frac{\delta}{\delta m } b$ and $\partial_{\tilde{x}} \frac{\delta}{\delta m } \sigma$ satisfying the following properties
    \begin{align}
    \begin{cases}
        | \partial_{\tilde{x}} \frac{\delta}{\delta m }b(u,x,\mu,a)(\tilde{u},\tilde{x}) - \partial_{\tilde{x}} \frac{\delta}{\delta m } b(u,x',\mu',a')(\tilde{u}',\tilde{x}') |\leq L \big( |x-x'| + |\tilde{x}- \tilde{x}'| + \Wc_2(\mu,\mu') \big), \\
         | \partial_{\tilde{x}} \frac{\delta}{\delta m }\sigma(u,x,\mu,a)(\tilde{u},\tilde{x}) - \partial_{\tilde{x}} \frac{\delta}{\delta m } \sigma(u,x',\mu',a')(\tilde{u}',\tilde{x}') |\leq L \big( |x-x'| + |\tilde{x}- \tilde{x}'| + \Wc_2(\mu,\mu') \big),
    \end{cases}
    \end{align}
    for every $u,\tilde{u} \in I$, $x,x',\tilde{x},\tilde{x}' \in \R^d$ and $\mu,\mu' \in \Mc_\lambda$ and 
    \begin{align}
         \big | \partial_{\tilde{x}} \frac{\delta}{\delta m }b(u,0, \lambda \otimes \delta_0,a)| + | \partial_{\tilde{x}} \frac{\delta}{\delta m } \sigma(u,0, \lambda \otimes \delta_0, a) | \leq  M \big(1+ |a| \big),
    \end{align}
    for every $u,\tilde{u} \in I$ and $a \in A$.
    \item [(3)] The maps $f$ and $g$ are differentiable with respect to $(x,a)$. Moreover, $\partial_x(f,g)$ and $\partial_a f$ are assumed to be uniformly bounded  $\lambda(\d u )-\text{a.e}$. Finally, the maps $(x,\mu,a) \mapsto \partial_x f(u,x,\mu,a)$, $\partial_a f(u,x,\mu,a)$ and $(x,\mu) \mapsto \partial _x g(u,x,\mu)$ are continuous for $\lambda-\text{a.e}$ $u \in I$.
    \item [(4)] The functions $f$ and $g$ admit Fréchet differentiable linear functional derivatives. Moreover, for any progressively measurable process $X$  such that $ \E \big[\underset{0 \leq t \leq T}{\text{ sup} } |X_t|^2 \big] < \infty$, the following quantities are uniformly bounded $\lambda(\d u)-\text{a.e}$
    \begin{align}
        \tilde{\E} \Big[ | \partial_{\tilde{x}} \frac{\delta}{\delta m} f(u,x,\mu,a)(\tilde{U},\tilde{X})|^2 \Big], \text{ and }  \tilde{\E} \Big[ | \partial_{\tilde{x}} \frac{\delta}{\delta m} g(u,x,\mu)(\tilde{U},\tilde{X})|^2 \Big].
    \end{align}
\end{enumerate}

\end{Assumption}

\begin{Proposition}\label{prop : pontryagin_formulation}
    Let Assumptions \ref{assumption : regularity_coefficients}, \ref{assumption : cost_functional_coefficients}, \ref{assumption2: regularity_coefficients} hold and assume furthermore that the Hamiltonian map $H$ defined in \eqref{eq : hamiltonian_def} is a convex function in its last variable, i.e. the map
    \begin{align}
        A \ni a \mapsto  H(u,x,\mu,y,z,a),
    \end{align}
    is convex for any $(u,x,\mu,y,z) \in I \times \R^d \times \Mc_\lambda \times \R^d \times \R^{d \times n }$. Let $(\alpha_t)_{0 \leq t \leq T}$ be an optimal control and $(X_t,Y_t,Z_t)_{0 \leq t \leq T}$ the respectively associated controlled state processes  and  adjoint processes defined in \eqref{eq : FBSDE_system_weak_formulation}. Then the optimal control problem is a point-wise minimizer of $H$, i.e. for any $a \in A$.
    \normalfont
    \begin{align}\label{eq : optimal_control}
        H(U,X_t, \P_{(U,X_t)}, Y_t,Z_t, \alpha_t) \leq H(U,X_t,\P_{(U,X_t)}, Y_t,Z_t, a) \quad \d t \otimes \d \P-\text{a.e}.
    \end{align}
\end{Proposition}

\begin{proof}
    As the proof is fairly similar to the one presented in \cite{kharroubi2025stochastic,cao2025graphon}, we present only the key steps in Appendix \ref{appendix : proof_optimality_alpha}.
\end{proof}

\begin{Remark}
\normalfont

Under the following stronger convexity assumptions, we can turn the necessary condition into a sufficient condition, namely we need to assume that
\begin{enumerate}
    \item [(1)] The map $\R^d \times \Mc_{\lambda} \ni (x,\mu) \mapsto g(U,x,\mu)$ is convex $\P-\text{a.s}$
    \item [(2)] The map $\R^d \times \Mc_{\lambda} \times A \ni (x,\mu,a) \mapsto H(U,x,\mu,Y_t,Z_t,a)$ is convex $\d t \otimes \d \P-\text{a.e}$.
\end{enumerate}
With this additional requirements and if $\alpha^{\star}$ satisfies \eqref{eq : optimal_control}, then $\alpha^{\star}$ yields an optimal control.
\end{Remark}

At this point, we are able to identify a potential optimal control $\alpha$ in the form of a progressively measurable map of $(U,X,Y,Z)$ as a minimizer of the Hamiltonian map. However, plugging the potential optimal control obtained from \eqref{eq : optimal_control} in the forward and adjoint equations leads to the study of a fully coupled FBSDE, namely \eqref{eq : FBSDE_system_weak_formulation}, for which we need to prove existence and unicity. Motivated by the numerical applications, we will concentrate ourselves to the subclass of linear dynamics and quadratic cost functional for which we will prove existence and uniqueness in the following section. In the existing literature, existence and uniqueness of the resulting FBSDE system can be proven under linear dynamics but for a slightly broader class of cost functional. We refer to \cite{kharroubi2025stochastic,carmona2015fbsde} for full details.

\subsubsection{Dynamic programming and HJB equation.}

Solution to \eqref{eq : value_function_time_t} can also be characterized by the dynamic programming method. We will restrict ourselves to give the main results that will be used to design algorithms. As we show for the case of the stochastic maximum principle, the proof relies on similar arguments and we refer to \cite{decrescenzo2024mean} for a rigourous analysis of the HJB equation (note that their HJB is stated on the space $L^2(I;\Pc_2(\R^d))$ but the analysis can be easily extended to our setting to the space $\Mc_{\lambda}$. This approach is based  on a recursive argument after defining the decoupled value function $V : I \times [0,T] \times \R^d \times \Mc_{\lambda}$ of problem which satisfies 
\begin{align}\label{eq : dynamic_programming_equation}
    V(t,U,X_t, \P_{(U,X_t)}) = \underset{\alpha \in \Ac}{\text{inf }} \E \Big[ \int_{t}^{t+h} f(U,X_s, \P_{(U,X_s)}, \alpha_s) \d s    + V(t+h, X_{t+h}, \P_{(U,X_{t+h})}) | \Fc_t \Big],
\end{align}
for any $t \in [0,T)$ and $h \in (0,T-t]$. and starting from the terminal condition $V(T,u,x,\mu) = g(u,x,\mu)$ for $(u,x,\mu)  \in I \times \R^d \times \Mc_{\lambda}$.
Following the result in \cite{decrescenzo2024mean} which can be adapted to our current setting (in view of the Pontryagin's Maximum principle) and assuming that  for any $(u,t,x,\mu,p,M)  \in I \times [0,T] \times \R^d \times \Mc_{\lambda} \times \R^d \times \R^{d \times d}$,
there exists a minimizer 
\begin{align}\label{eq : argmmin_control}
    \hat{a}(u,x,\mu,p,M)  \in \underset{a \in A}{\text{arg min }} \mrH(u,x,\mu,p,M,a) 
\end{align}
where the Hamiltonian map $\mrH$ is defined as
\begin{align}
    \mrH(u,x,\mu,p,M,a) := b(u,x,\mu,a) \cdot p + \frac{1}{2} \sigma \sigma^{\top}(u,x,\mu,a): M + f(u,x,\mu,a),
\end{align}
one can show by sending formally $h \to 0$ in \eqref{eq : dynamic_programming_equation} that 
\begin{align}\label{eq : HJB equation}
    \begin{cases}
        \partial_t V(t,u,x,\mu) + b\big(u,x,\mu, \hat{a}(u,x,\Uc(u,t,x,\mu), \partial_x \Uc(u,t,x,\mu) \big) \cdot \partial_x V(u,t,x,\mu) \\ + \frac{1}{2} \sigma \sigma^{\top}\big(u,x,\mu, \Uc(u,t,x,\mu), \partial_x \Uc(t,u,x,\mu)\big): \partial^2_{x} V(t,u,x,\mu) \\
        +\E_{(U,\xi)\sim \mu} \Big[  b \big(U,\xi,\mu, \hat{a}(U,\xi, \Uc(t,U,\xi,\mu), \partial_x \Uc(t,U,\xi,\mu) \big) \cdot \partial_{\tilde{x}} \frac{\delta}{\delta m} V(t,u,x,\mu)(U,\xi)  \\+ \frac{1}{2} \sigma \sigma^{\top}(t,U,\xi,\mu, \hat{a}(U,\xi,\Uc(t,U,\xi,\mu), \partial_x \Uc(t,U,\xi,\mu) \big) :\partial^2_{\tilde{x}} \frac{\delta}{\delta m} V(t,u,x,\mu)(U,\xi) \Big] \\
        +f \big(u,x,\mu,\hat{a}(u,t,x,\Uc(t,u,x,\mu),\partial_x \Uc(t,u,x,\mu) \big) + \frac{1}{2} \Uc^2(t,u,x,\mu)=0\\
        V(T,u,x,\mu) = g(u,x,\mu),
\end{cases}
\end{align}
where the master field $\Uc : I \times [0,T] \times \R^d \times \Mc_{\lambda}$ is defined as 
\begin{align}
    \Uc(t,u,x,\mu) &:= \partial_x V(t,u,x,\mu) + \E_{(U,\xi) \sim \mu} \big[ \partial_{\tilde{x}} \frac{\delta}{\delta m} V(t,U,\xi,\mu)(u,x) \big] \\
    &=\partial_{\tilde{x}} \frac{\delta}{\delta m}v(t,\mu)(u,x), \quad \text{ with $v(t,\mu) := \E_{(U,\xi) \sim \mu} \big[ V(t,U,\xi,\mu) \big]$.}
\end{align}
If the optimal feedback control $\hat{a}$ obtained in \eqref{eq : argmmin_control} is Lipschitz in all its variables, then we get from \eqref{assumption : regularity_coefficients} an optimal feedback control given by
\begin{align}
    \mathfrak{a}^{\star}(t,u,x,\mu)  = \hat{a}\big(u,x,\mu, \Uc(t,u,x,\mu), \partial_x \Uc(t,u,x,\mu) \big), \quad (t,u,x,\mu) \in [0,T]\times I \times \R^d \times \Mc_{\lambda}.
\end{align}
\subsection{The class of linear quadratic control problem}\label{subsec : linear_quadratic_control_problem}

For sake of simplicity, we will present the computations in the case of a constant volatility term $\sigma$ but the whole analysis could be performed by relaxing this hypothesis. 

\noindent Given $\alpha \in \Ac$, and $\xi$ an admissible initial condition and the set of controls $A \subset \R^m$ for $m \in \N^{\star}$, we consider the controlled state process satisfying the SDE
\begin{align}
\begin{cases}
    \d X_t &= \Big[A(U) + B(U) X_t + \E_{(\tilde{U},\tilde{X}_t) \sim \P_{(U,X_t)}} \big[ G_B(U,\tilde{U}) \tilde{X}_t \big] + C(U) \alpha_t \Big] \d t + \sigma(U) \d W_t, \\
    X_0 &= \xi,
\end{cases}
\label{eq:XForwardLQ}
\end{align}
where $A \in L^2(I; \R^d)$, $B \in L^{\infty}(I ;\R^{d \times d})$, $C \in L^{\infty}(I ;\R^{d \times m})$, $G_B \in L^2_{\text{sym}}(I \times I ; \R^{d \times d})$ and $\sigma \in L^2(I;\R^{d \times n})$ where the spaces $L^2$, $L^{\infty}$ and $L^2_{\text{sym}}$ have been introduced in \ref{subsec:  notations_lq_example}.

The cost functional $J$ is given by 
\begin{align}
    J(\alpha) &= \E \bigg[ \int_{0}^{T} \Big( Q(U) \big( X_t - \tilde{\E} \big[ \tilde{G}_Q(U,\tilde{U}) \tilde{X}_t]\big) \cdot  (X_t   - \tilde{\E} \big[ \tilde{G}_Q(U,\tilde{U}) \tilde{X}_t]\big)  + \alpha_t^{\top} N(U) \alpha_t \Big) \d t \\
    &\qquad \qquad+  H(U) \big( X_T - \tilde{\E} [  \tilde{G}_H(U,\tilde{U}) \tilde{X}_T ] \big) \cdot \big( X_T -  \tilde{\E} [  \tilde{G}_H(U,\tilde{U}) \tilde{X}_T ] \big)    \bigg],
    \label{eq:JLQ}
\end{align}
where $Q,H \in L^{\infty}(I ;\S^d_{+})$ where $\S^d_{+}$ denotes the set of positive symmetric matices over $\R^{d \times d}$, $\tilde{G}_Q, \tilde{G}_H \in L^2_{\text{sym}}(I \times I ; \R^{d \times d})$ and $N \in L^{\infty}(I ; \S^m_{ >+})$ where $\S^m_{> +}$ denotes the set of positive definite symmetric matrices over $\R^{m \times m}$. Moreover, we suppose that there exists $c > 0$ such that 
\begin{align}
    \langle N(U)y, y \rangle \geq c |y|^2, \quad \P-\text{a.s}, \quad \forall y \in \R^m.
\end{align}
We also suppose that the operators $T_{\tilde{G}_Q}$ and $T_{\tilde{G}_H}$ are positive symmetric operators (where the operator notations have been introduced in  \ref{subsec:  notations_lq_example}).
\begin{Remark}\label{rmk : cross_terms_control_state}
\normalfont
     Under the assumptions on the model coefficients, we notice that for any admissible initial condition $\xi$ and any $\alpha \in \Ac$
    \begin{align}\label{eq : positivity_cost_functional}
        J(t,\xi,\alpha) \geq 0, \quad \forall t \in [0,T].
    \end{align}
This assumption is required as it will help us to derive an a-priori estimate on the triangular Riccati system which will arise from this control problem. However, in the numerical examples, we may add cross-product terms between the state and the control namely terms in the form
$2 \alpha_t^{\top} I(U) X_t$  and $ \tilde{\E} [\alpha_t^{\top}  G_I(U,\tilde{U}) \tilde{X}_t]$ where $I \in L^{\infty}(I;\R^{m \times d})$ and $G_I \in L^2( I \times I ; \R^{m \times d})$ even if we are not able to prove the existence and uniqueness of the associated Riccati equations in this setting.
\end{Remark}

Recalling the FBSDE system in \eqref{eq : FBSDE_system_weak_formulation}, and following Proposition \ref{prop : pontryagin_formulation}, an optimal control $\alpha^{\star}$ should satisfy 
\begin{align}
    \alpha_t^{\star} = - \frac{1}{2} N(U)^{-1} C(U)^{\top} Y_t ,
\end{align}
and following the notion of derivative introduced in \eqref{def :derivative_notion}, we end up with
\footnotesize
\begin{align}\label{eq : FBSDE_linear_quadratic}
\begin{cases}
    \d X_t &= \Big( A(U) +  B(U) X_t  + \tilde{\E} \Big[  G_B(U,\tilde{U}) \tilde{X}_t \Big]   - \frac{1}{2}  C(U) N(U)^{-1} C(U)^{\top} Y_t \Big) \d t   \\
    &\quad +\sigma (U) \d W_t \\
    \d Y_t  &= \Big( C_Y(U)  Y_t +  2 C_X(U) X_t + 2 \tilde{\E} \Big[ \Psi_X(U,\tilde{U}) \tilde{X}_t \Big] + \tilde{\E} \Big[ \Psi_Y(U,\tilde{U}) \tilde{Y}_t \Big] \Big) \d t  + Z_t \d W_t, \\
    X_0 &= \xi, \\
    Y_T &= 2 \Big( H(U) X_T + \tilde{\E} \big[ G_H(U,\tilde{U}) \tilde{X}_T] \big] \Big),
\end{cases}
\end{align}
\normalsize
\begin{align}
\begin{cases}
    C_Y(u) &=  B(u)^{\top} ,\\
    C^X_t(u) &= Q(u)   ,\\
    \Psi^X_t(u, \tilde{u}) &=  G_Q(u,\tilde{u}), \\
    \Psi^Y_t(u,\tilde{u}) &= G_B^{\star}(u,\tilde{u}) - G_I^{\star}(u,\tilde{u}) N_t(\tilde{u})^{-1} \Gamma_t(\tilde{u}) , \\
    G_Q(u,\tilde{u}) &=  (\tilde{G}_Q \circ Q \circ \tilde{G}_Q)(u,\tilde{u}) - \big( Q(u) + Q(\tilde{u}) \big) \tilde{G}_Q(u,\tilde{u}), \\
    G_H(u,\tilde{u}) &= (\tilde{G}_H \circ H \circ \tilde{G}_H)(u,\tilde{u}) - \big( H(u) + H(\tilde{u}) \big) \tilde{G}_H(u,\tilde{u})
\end{cases}
\end{align}
where we introduced the notations (see Appendix \ref{subsec:  notations_lq_example} for the $\circ$ and $^{\star}$ notations).

\begin{Theorem}\label{thm : optimal_control_form}
    Under the assumptions on the model coefficients,  the FBSDE \eqref{eq : FBSDE_linear_quadratic} is uniquely solvable and the optimal control $\hat{\alpha}$ is then given as 
    \begin{align}
    \label{eq:opContLQ}
        \hat{\alpha}_t = - \frac{1}{2} N(U)^{-1}  C(U)^{\top} Y_t, \quad \forall t \in [0,T].
    \end{align}
    Moreover, we have the following form for the adjoint process $Y=(Y_t)_{0 \leq t \leq T}$
    \begin{align}\label{eq : form_adjoint_process}
        Y_t = K_t(U) X_t + \tilde{\E}_{(\tilde{U},\tilde{X}_t) \sim \P_{(U,X_t)}} \big[ \bar{K}_t(U,\tilde{U}) \tilde{X}_t \big] + \Lambda_t(U),
    \end{align}
    where $K \in \Cc_1 \big([0,T];L^{\infty}([0,T] ;  \S^d_+)\big)$ is the unique solution of the infinite dimensional Riccati equations 
    \begin{align}\label{eq :standard_Riccati}
    \begin{cases}
       \dot{K}_t(u) + \Phi(u,K_t(u) ) -  U(u,K_t(u)) N(u)^{-1}  U(u,K_t(u)) \; = 0  \\
        K_T(u) \; = \;  H(u)
    \end{cases}
    \end{align}
    for $\lambda(\d u )-\text{a.e}$ and where we introduced the measurable maps $\Phi : I \times  \S^d_{+} \to \R^{d \times d}$ and $U : I \times \S^d_{+} \to \R^{\d \times d}$ as
    \begin{align}
    \begin{cases}
        \Phi(u,\kappa) &:=  B(u)^{\top} \kappa + \kappa^{\top} B(u), \\
        U(u,\kappa) &:= C(u)^{\top} \kappa.
    \end{cases}
    \end{align}    
    Moreover, $\bar{K} \in \Cc_1\big([0,T],L^2_{\text{sym}}(I \times I, \R^{d \times d})\big)$ is the unique solution to the abstract Riccati equation on the Hilbert space $L^2(I \times I ; \R^{d \times d})$
    \begin{align}\label{eq : abstract_Riccati}
    \begin{cases}
        \dot{\bar{K}}_t + \Psi(t,K_t, \bar{K}_t) - L(t,K_t, \bar{K}_t) = 0, \\
        \bar{K}_T = G_H,
    \end{cases}
    \end{align}
    where we introduced the measurable map $\Psi,L : [0,T]  \times L^{\infty}(I ;\S^d_{+})\times L^2(I \times I ; \R^{d \times d }) \to L^2(I \times I ; \R^{d \times d})$ and $V : L^2(I \times I ; \R^{d \times d}) \to L^2(I \times I ; \R^{d \times d})$ defined as  
    \begin{align}
    \begin{cases}
        \Psi(t, K_t, \bar{k})(u,v) &:= K_t(u) G_B(u,v) + G_B^{\star}(u,v) K_t(v) + B(u)^{\top} \bar{k}(u,v) + \bar{k}^{\star}(u,v) B(v) \\\
        &\quad + \big( \bar{k}^{\star} \circ  G_A \big)(u,v) + \big( G_A^{\star} \circ \bar{k})(u,v)  + G_Q(u,v), \\
        L(t,K_t, \bar{k})(u,v) &:=U(u,K_t(u))  N(u)^{-1} V(u,v)(\bar{k}) - V(u,v)^{\star}(\bar{k}) N(v)^{-1} U(v,K_t(v)) \\
        &\quad- \big( V^{\star} \circ N \circ V)(u,v), \\
        V(u,v)(\bar{k}) &:= C(u)^{\top} \bar{k}(u,v),
    \end{cases}
    \end{align}
    and where $\Lambda \in \Cc_{1}([0,T];L^2(I ;\R^d) \big)$ is the unique solution to the linear equation on the Hilbert space $L^2(I ;\R^d)$
    \begin{align}
    \begin{cases}
        \dot{\Lambda}_t + F(t,\Lambda_t) - M(t,\Lambda_t) = 0, \\
        \Lambda_T = 0.
    \end{cases}
    \label{eq:Lambda}
    \end{align}
    where we introduced the measurable maps $F,M : [0,T] \times L^2(I ;\R^d) \to L^2(I;\R^d)$ defined as
    \begin{align}
    \begin{cases}
        F(t,\lambda)(u) &:= B(u)^{\top} \lambda(u) + T_{G_A^{\star}}(\lambda)(u)  + K_t(u) A(u) + T_{\bar{K}_t}(A)(u),\\
        M(t,\lambda)(u) &:= U(u,K_t(u))^{\top} N(u)^{-1} C(u)^{\top} \lambda(u) + T_{V^{\star}}(NC^{\top} \lambda)(u).
    \end{cases}
    \end{align}
\end{Theorem}
\begin{proof}
    The proof of this result is discussed in Appendix \ref{proof: thm_optimal_control_form}.
\end{proof}

\begin{Remark}

\normalfont

\noindent    In fact, the Riccati for $K$ \eqref{eq :standard_Riccati} in Theorem \ref{thm : optimal_control_form} is standard as it can be solved $u$ by $u$ $\lambda(\d u )$-a.e. Moreover, a known formula is known for scalar Riccati equations (see Equation (2.50) in \cite{carmona_probabilistic_2018a}) and we will rely on this for the resolution of $\bar{K}$ in \eqref{eq : abstract_Riccati}.
\end{Remark}

\subsection{Algorithms}

We now illustrate the methodology introduced in the previous sections to solve this class of non exchangeable mean field control problems. We will, in the spirit of the methods developed in \cite{pham2022mean}, develop two  methods to solve the associated control problem. The first one will rely on a global learning method 
as in \cite{HanE16} which we will refer as Deep Graphon while the other one will take advantage of the FBSDE reformulation of the control problem which we will refer as Deep Graphon BSDE in line with \cite{Jentzen17}.   
Notice that even if we apply only two algorithms for the resolution, all kinds of algorithms developed in \cite{pham2022mean} can be used.
Moreover, we give a Riccati solver for the abstract Riccati term $\bar{K}$ over the Hilbert space $L^2(I \times I ; \R^{d \times d})$.

\subsubsection{Riccati solver}
The Riccati equations depend on the values taken by \( U \) and involve integration with respect to a random variable \( \tilde U \) independent with respect to \( U \) and with the same law.
To make the system solvable, we must assume that the interaction function and the coefficients \( Q \) and \( H \) are such that \( G_Q(u,\tilde u) \) and \( G_H(u,\tilde u) \) can be computed analytically. This restricts the class of admissible interaction functions and, for example, prevents us from using the function \( G_2 \) in \eqref{eq : G2_map}.

We are then left to numerically integrate \eqref{eq :standard_Riccati}, \eqref{eq : abstract_Riccati}, and \eqref{eq:Lambda} for \((U_i)_{i \in \llbracket 1,N \rrbracket }\), a given set of samples drawn from $ \Uc([0,1])$.
Algorithm~\ref{algo:ricat} provides the procedure used to solve the Riccati equations.

Note that, to obtain a highly accurate reference solution, we use a very small time step, significantly smaller than the one for  in the neural-network-based algorithm.

\begin{algorithm}[H]
 {\bf Input}: 
 $(u_i)_{i \in \llbracket 1,N \rrbracket}$ samples from $\Uc([0,1])$, $L$ number of time steps, $\Delta t$ the time step

Initialize : for $(i,j) \in \llbracket 1, N \rrbracket^2$
\begin{align}
    \bar K_L(u_i,u_j)=G_H(u_i,u_j) , \quad
    K_L(u_i) = H(u_i) , \quad
    \Lambda_L(u_i)= 0.
\end{align}
\For{ $l=L-1,\dots,0$, }{
Calculate $K$ term: for $ i \in \llbracket 1, N \rrbracket$
\begin{align}
\Phi(u_i) = & B(u_i)^{\top} K_{l+1}(u_i) + K_{l+1}(u_i)^\top B(u_i), \\
        U(u_i) =& C(u_i)^{\top} K_{l+1}(u_i), \\
    K_l(u_i)= & K_{l+1}(u_i) + \Delta t \big( \Phi(u_i) -  U(u_i) N(u_i)^{-1}  U(u_i) \big).
\end{align}
Calculate $\Lambda$ term: for $ i \in \llbracket 1, N \rrbracket$
{\small
\begin{align}
 F(u_i) =& B(u_i)^{\top} \Lambda_{l+1}(u_i) + \frac{1}{M}\sum_{m=1}^N G_A(u_m,u_i)\Lambda_{l+1}(u_m)   + K_{l+1}(u_i) A(u_i) + \frac{1}{M}\sum_{m=1}^N  \bar{K}_{l+1}(u_i,u_m)A(u_m),\\
M(u_i) = &K_{l+1}^\top(u_i) C(u_i) N(u_i)^{-1} C(u_i)^{\top} \lambda_{l+1}(u_i) + \frac{1}{M}\sum_{m=1}^N  
\bar K_{l+1}(u_m,u_i) C(u_i) N(u_m) C^\top(u_m) \Lambda_{l+1}(u_m), \\
\Lambda_l(u_i)= & \Lambda_{l+1}(u_i) + \Delta t \big( F(u_i) - M(u_i) \big).
\end{align}
}
Calculate $\bar K$ term: for $(i,j) \in \llbracket 1, N \rrbracket^2$
{\small
\begin{align}
   \Psi(u_{i},u_{j}) =& K_t(u_i) G_B(u_i,u_j) + G_B(u_j,u_i) K_t(u_j) + B(u_i)^{\top} \bar{K}_{l+1}(u_i,u_j) + \bar{K}_{l+1}(u_j,u_i) B(u_j)  + \\
   & \frac{1}{M}\sum_{m=1}^N \bar K_{l+1}(u_m,u_i)G_A(u_m,u_j) + \frac{1}{M}\sum_{m=1}^N G_A(u_m,u_i) \bar K_{l+1}(u_m,u_j) + G_Q(u_i,u_j),\\
   V(u_i,u_j)= & C(u_i)^{\top} \bar{K}_{l+1}(u_i,u_j), \\
   L(u_i,u_j)=& U(u_i,K_t(u_i))  N(u_i)^{-1} V(u_i,u_j) - V(u_j,u_i) N^{-1}(u_j) U(u_j,K_t(u_j)) - \\
   & \frac{1}{M} \sum_{m=1}^N V(u_m,u_i)N(u_m) V(u_m,u_j), \\
    \bar K_l(u_i,u_j)=& \bar K_{l+1}(u_i,u_j) + \Delta t \Big( \Psi(u_{i},u_{j}) - L( u_{i},u_{j}) \Big).
\end{align}
}
}
\caption{Riccati solver for the non exchangeable mean field terms}\label{algo:ricat}
\footnotesize
\end{algorithm}

\begin{Remark}
    In the linear quadratic case treated below, the values for $K$ and $\Lambda$ can be known analytically in some special cases, therefore we are left to solve only \eqref{eq : abstract_Riccati}.
\end{Remark}
Once $\Lambda_l(u_i)$, $K_l(u_i)$, $\bar K_l(u_i,u_j)$, $(i,j) \in \llbracket 1, N \rrbracket^2$ are computed for $l \in \llbracket 0,L \rrbracket$, we can recover an estimation of $X_{l \Delta t}(u_i)$ for $i \in \llbracket 1, N \rrbracket$ solving the Euler scheme and we  estimate $Y_{l \Delta t}$:
\begin{align}\label{eq : form_adjoint_process}
        Y_{l\Delta}(u_i) \simeq K_{l \Delta t}(u_i) X_{l \Delta t}(u_i) + \frac{1}{M} \sum_{m=1}^M \bar{K}_{l \Delta t }(u_i,u_m) X_{l \Delta t}(u_m)  + \Lambda_{l \Delta t}(U_i),
    \end{align}
    and then recover the optimal control using \eqref{eq:opContLQ}.

\subsubsection{Deep Graphon}
The Deep Graphon algorithm allows us to compute the optimal control 
\( \alpha \) associated with the optimization problem 
\eqref{eq:XForwardLQ}-\eqref{eq:JLQ} for all \( X_0 \) sampled from a 
probability space $\M$.
It uses two networks:
\( \mathcal{T}^{\theta_1} \), parametrized by \( \theta_1 \), and 
\( \mathcal{B}^{\theta_2} \), parametrized by \( \theta_2 \), both taking
values in \( \mathbb{R}^{r \times m} \), where \( m \) denotes the 
dimension of the control. The first network $\Tc^{\theta_1}$ inputs  $(t,u, x) \in [0,T] \times [0,1] \times \R^d$ and the second ones inputs $(t,y) \in [0,T] \times \R^J$ where $y$ represents the moment vector of the distribution. We note $\theta := (\theta_1,\theta_2)$.

To simplify the notation, we assume in Algorithm~\ref{algo:DeepGraphon}
that at each iteration a single distribution is sampled to initialize 
\( X_0 \). The batched version of the algorithm is straightforward.  

\begin{algorithm2e}[H]
\DontPrintSemicolon 
\SetAlgoLined 
\footnotesize
\caption{Deep Graphon algorithm for non exchangeable mean-field algorithm }\label{algo:DeepGraphon}.
\DontPrintSemicolon 
\SetAlgoLined 

{\bf Input}: $N$ batch size, $L$ number of time steps, $\theta$ parameters of the networks, $\rho$ learning rate

\For{each epoch $e$}{
    Sample $(U_{n}, X_{n}) \sim \Uc([0,1]) \otimes \mu$,  $n \in \llbracket 1, N \rrbracket$ where  $\mu$ is sampled in $\M $. \;
    Initialize cost: $\hat C_n = 0$ for  $n \in \llbracket 1, N \rrbracket$, $\Delta t = \frac{T}{L}$ . \;
    \For {$ l=1 \dots,L$}
    {
       Calculate $\tilde W = (\frac{1}{N} \sum_{n=1}^N  |X_{n}|^j)_{j \in \llbracket 1,J \rrbracket}$, \; 
        Calculate $\hat W= \Bc^{\theta_2}(l\Delta t,\tilde W)$, \; 
        Calculate $W_{n}= \Tc^{\theta_1}(l\Delta t, U_{n},X_{n})$ for $n \in \llbracket 1, N \rrbracket$,\;
        Estimate the control  $\alpha_{n} = \sum_{k=1}^r W_{n,r} \hat W_{r} $  for $n \in \llbracket 1, N \rrbracket$.  \;      
        Update cost : for $n \in \llbracket 1, N \rrbracket$,
    \begin{align}
     d_n= & Q(U_n) \big( X_n - \frac{1}{N} \sum_{m=1}^{N} \tilde{G}_Q(U_n,U_m) X_m\big), \\
     e_n= & (X_n   -  \frac{1}{N} \sum_{m=1}^{N} \tilde{G}_Q(U_n,U_m) X_m\big) , \\
    \hat C_n= &  \hat C_n + \Delta t \big[ d_n \cdot e_n\big ] + \alpha_n^{\top} N(U_n) \alpha_n.
\end{align}
 Sample $ g_n \sim \mathcal N(0,I_{d})$ for $n \in \llbracket 1, N \rrbracket$. \;
       Estimate $X$ drift : for $n \in \llbracket 1, N \rrbracket$,
        $$ T^X_n =  A(U_n) + B(U_n) X_n +  \frac{1}{N} \sum_{m=1}^{N}  G_B(U_n,U_m) X_m + C(U_n) \alpha_n.  $$
        Update the state : for $n \in \llbracket 1, N \rrbracket$, 
        $
        X_n= X_n + \Delta t T^X_n + \sigma(U_n) g_n \sqrt{\Delta t}.
        $
    }
    Add terminal cost : for $n \in \llbracket 1, N \rrbracket$,
   \begin{align}
   \hat C_n= &  \hat C_n + H(U_n) \big( X_n - \frac{1}{N} \sum_{m=1}^{N}   \tilde{G}_H(U_n,U_m) X_m \big) \cdot \big( X_n -   \frac{1}{N} \sum_{m=1}^{N} \tilde{G}_H(U_n,U_m) X_m  \big).  
   \end{align}
   Cost function to minimize:
    $J(\theta) = \frac{1}{N} \sum_{n=1}^N \hat C_n$. \;
 $\theta =\theta - \rho \nabla J(\theta)$    
}
\end{algorithm2e}

\subsubsection{Deep Graphon BSDE }
The Deep  Graphon  BSDE Algorithm  \ref{algo:DeepGraphon} allows us to calculate the optimal  $Z=(Z_t)_{0 \leq t \leq T}$ and the cost function $Y_0$  for all $X_0$  sampled in a probability space $\M$ by solving the FBSDE \eqref{eq : FBSDE_linear_quadratic}.\\
It  uses four different networks:
\begin{itemize}
    \item Two networks to approximate $Y_0$ :  $\phi^{\zeta_1}$  with parameter $\zeta_1$  with input in $I \times \R^d$ and output in  $\R^{r \times d}$, $\psi^{\zeta_2}$  with parameter $\zeta_2$  with input in $\R^J$ and output in  $\R^{r \times d}$. We note $\zeta =(\zeta_1,\zeta_2)$.
    \item Two others neural network, $\Tc^{\theta_1}$ parametrized with $\theta_1$ and  $ \Bc^{\theta_2}$ parametrized by $\theta_2$, both  with output in $\R^{r \times d}$,  are used  to learn the $Z$ term in the BSDE.
    The first $\Tc^{\theta_1}$ inputs  $(t,u, x) \in [0,T] \times [0,1] \times \R^d$ and the second ones inputs $(t,y) \in [0,T] \times \R^J$. We still note $\theta= (\theta_1,\theta_2)$.
\end{itemize}

We present  Algorithm \ref{algo:DeepBSDEGraphon}.

\begin{algorithm2e}[H]
\DontPrintSemicolon 
\SetAlgoLined 
\footnotesize
\caption{Deep Graphon BSDE algorithm for non exchangeable mean-field algorithm }\label{algo:DeepBSDEGraphon}.
\DontPrintSemicolon 
\SetAlgoLined 

{\bf Input}: $N$ the batch size, $L$ number of time steps, $\zeta$ ,$\theta$, $\rho$ learning rate

\For{each epoch $e$}{
    Sample $(U_{n}, X_{n}) \sim \Uc([0,1]) \otimes \mu$,  $n \in \llbracket 1, N \rrbracket$ where  $\mu$ sampled in $\bar M $, \;
    Calculate $\tilde W = (\frac{1}{N} \sum_{n=1}^N  |X_{n}|^j)_{j \in \llbracket 1, J \rrbracket} $. \;
    Calculate $\hat W= \psi^{\zeta_2}(\tilde W)$. \; 
    Calculate $W_{n}= \phi^{\zeta_1}( U_{n},X_{n})$ for $n \in \llbracket 1, N \rrbracket$. \;
    Initialize $Y_{n}= \sum_{k=1}^r W_{n,r} \hat W_{r} $ for  $n \in \llbracket 1, N \rrbracket$ , $\Delta t = \frac{T}{L}$.\;
    \For{$ l=1 \dots,L$}
    {
        Calculate $\tilde W = (\frac{1}{N} \sum_{n=1}^N  |X_{n}|^j)_{j \in \llbracket 1, J \rrbracket} $. \; 
        Calculate $\hat W= \Bc^{\theta_2}(l\Delta t,\tilde W)$. \; 
        Calculate $W_{n}= \Tc^{\theta_1}(l\Delta t, U_{n},X_{n})$ for $n \in \llbracket 1, N \rrbracket$. \;
        Estimate the control  $Z_{n} = \sum_{k=1}^r W_{n,r} \hat W_{r} $  for $n \in \llbracket 1, N \rrbracket$.  \;
        Estimate  $Y$ trend: for $n \in \llbracket 1, N \rrbracket,$
        $$
            T_{n}^Y = C_Y(U_{n})  Y_{n} +  2 C_X(U_n) X_n + \frac{2}{N}  \sum_{m=1}^N \Psi_X(U_n,U_m) X_m + \frac{1}{N}\sum_{m=1}^N \Psi_Y(U_n,U_m) Y_m. 
             $$
        Sample $ g_n \sim \mathcal N(0,I_d)$ for $n \in \llbracket 1, N \rrbracket$. \;
        Update  the value function : for $n \in \llbracket 1, N \rrbracket$,
        $\hat Y_{n} =  Y_{n} + T_n^Y \Delta t + Z_{n} g_n \sqrt{\Delta t}$.  \;
        Estimate $X$ drift : for $n \in \llbracket 1, N \rrbracket$,
        $$ T^X_n =A(U_n) +  B(U_n) X_n  + \frac{1}{N} \sum_{m=1}^N G_B(U_n,U_m)) X_m   - \frac{1}{2}  C(U_n) N(U_n)^{-1} C(U_n)^{\top} Y_n.$$
        Update the state : for $n \in \llbracket 1, N \rrbracket$, 
        $
        X_n= X_n + \Delta t T^X_n + \sigma(U_n) g_n \sqrt{\Delta t}.
        $\;
        Store the value function $Y_n = \hat Y_n$ for $n \in \llbracket 1, N \rrbracket$. 
        }
        Calculate target: for $n \in \llbracket 1, N \rrbracket$,
        $
         \hat T_n = 2 \Big( H(U_n) X_n + \frac{1}{N} \sum_{m=1}^N G_H(U_n,U_m) X_m \big). 
        $\;
        Calculate the error
        $J(\zeta,\theta)= \frac{1}{N} \sum_{j=1}^N \Big( \hat T_n - Y_n \Big)^2$.\; 
    $(\zeta,\theta) = (\zeta,\theta) - \rho \nabla J(\zeta,\theta)$.
 }
\end{algorithm2e}

\subsection{Numerical experiments}

We present several numerical experiments illustrating the optimal control of non-exchangeable mean field systems. We first consider a linear–quadratic control problem arising in finance. We then turn to a more complex example, illustrating our approach in a non-toy, nonlinear and non-quadratic setting.
In the sequel we assume that \(X_0\) is sampled from the space of 
distributions \(\M\), generated randomly as a mixture of three 
Gaussian laws.  
At each iteration of the gradient descent algorithm, we sample for 
\(k \in \llbracket 1, 3 \rrbracket\),
\( W_k \sim \mathcal{U}([0,1]) \) and 
\( (\mu_k, \sigma_k) \sim \mathcal{U}([0,1])^2 \),
and then for any \(n \in \llbracket 1, N \rrbracket\) we compute
\[
    X_{0,n} = 
    \frac{1}{\sum_{k=1}^3 W_k}
    \sum_{k=1}^3 W_k\, Z_{k,n},
    \qquad \text{with } 
    Z_{k,n} \sim \mathcal{N}(\mu_k, \sigma_k^2).
\]

All results are obtained using a classical feedforward neural network with 
3 hidden layers of 10 neurons  or the spline KAN with 2 hidden layers of 10 neurons with 5 grid meshes. We take the $G_1$ interaction function. 
In all the tests, we take $r=10$ for the DeepONet network.
Since we must discretize both in time and with a high number of 
trajectories to accurately represent the distributions, GPU memory 
becomes the limiting factor.

For all experiments we use the ADAM gradient descent algorithm with a 
learning rate of \(0.001\), running for \(80{,}000\) iterations.

The networks are trained using \(N\) trajectories to sample distributions 
with Euler schemes discretized using \(L\) time steps.  
Then, for a given distribution \(\mu \in \M\), we estimate the cost 
function \(C^{\text{Alg}}(\mu)\) (depending on the resolution method) associated with 
the optimal control obtained (for example using the time-discretized 
version of \eqref{eq:JLQ} in the linear–quadratic case).  
This cost is estimated using a distribution sampled with the \(N\) 
trajectories generated from the samples of \(X_0\), \(U\), and the 
Brownian motions discretized with \(L\) time steps. In the two examples tested below, we either have:
\begin{itemize}
    \item an estimate of the optimal control and cost function 
          \(C^{\text{anal}}(\mu)\) from the Riccati equation using the 
          same samples of \(X_0\), \(U\), and the Brownian motion,
    \item or an analytical control that allows us to estimate the cost 
          function \(C^{\text{anal}}(\mu)\) with the same samples. 
\end{itemize}
Sampling \(1000\) distributions \((\mu_m)_{1 \leq m \leq 1000}\), we obtain 
\( \big(C^{\text{Alg}}(\mu_m) \big)_{m \in \llbracket 1, 1000 \rrbracket}\) and 
\((C^{\text{anal}}(\mu_m))_{m \in \llbracket 1, 1000 \rrbracket}\), and we can report different 
error measures whenever \(\text{Alg} = DG\) (Deep  Graphon) and 
\(\text{Alg} = BSDE\) (Deep Graphon BSDE)
\begin{align}
\begin{cases}
    E^{\text{abs}}_{\text{Alg}} &:= 
           \frac{1}{1000} \sum_{m=1}^{1000} 
           \left| C^{\text{Alg}}(\mu_m)-C^{\text{anal}}(\mu_m)\right|, \\
    E^{L_2}_{\text{Alg}} &:= 
           \frac{1}{1000} \sum_{m=1}^{1000} 
           \left| C^{\text{Alg}}(\mu_m)-C^{\text{anal}}(\mu_m)\right|^2, \\
    E^{\text{Sup}}_{\text{Alg}} &:=  
           \max_{m \in \llbracket 1,1000 \rrbracket} 
           \left| C^{\text{Alg}}(\mu_m)-C^{\text{anal}}(\mu_m)\right|.
\end{cases}
\end{align}

\subsubsection{A systemic risk model}

The following example is an extension of the mean-field systemic risk model introduced first in \cite{carmona2015systemic} to the case of heterogeneous banks. The following model representing the log monetary reserve of each bank $u \in I$ as the process $(X_t|U=u)_{0 \leq t \leq T}$
\begin{align}
\begin{cases}
        \d X_t &=  \Big( \kappa(U) ( \tilde{\E}_{(\tilde{U},\tilde{X}_t)} \big[ \tilde{G}_{\kappa}(U,\tilde{U}) \tilde{X}_t \big] -X_t) + \alpha_t \Big) \d t  + \sigma(U) \d W_t, \\
        X_0 &= \xi, \\
\end{cases}
\end{align}
where $\kappa \in L^{\infty}(I;\R_{-})$ acts like a mean reversion term, $\sigma \in L^2(I;\R)$ measures the volatility of the bank reserve and  $\tilde{G}_{\kappa} \in L^2(I \times I ; \R)$ is a graphon, i.e a measurable, bounded and symmetric map from $I \times I$ to $\R_+$ measuring the rate of borrowing between banks.  Moreover, 
 $\alpha=(\alpha_t)_{0 \leq \leq T}$ is the control rate of borrowing/ lending to a central bank that aims to minimize the functional cost
\begin{align}\label{eq : cost_functional_systemic_risk}
    J(\alpha) =  \E \Big[ \int_{0}^{T} f \big(U,X_t,\P_{(U,X_t)}, \alpha_t \big) \d t  + g\big(U,X_T, \P_{(U,X_T)}\big) \Big],
\end{align}
where the running and terminal cost functions are given by
\begin{align}
\begin{cases}
    f(u,x,\mu,a) &:=  \eta  \Big( x - \int_{I \times \R^d} \big(\tilde{G}_{\eta}(u,v) \tilde{x} \big)\mu(\d v , \d \tilde{x}) \Big)^2 + a^2 +  q  a \big( x  - \int_I \tilde{G}_q(u,v) \tilde{x} \mu(\d v , \d x) \big), \\
    g(u,x,\mu) &:= r \Big(x - \int_{I \times \R^d}  \big( \tilde{G}_r(u,v) \tilde{x} \big) \mu(\d v , \d \tilde{x}) \Big)^2,
\end{cases}
\end{align}
for some constants $\eta \geq 0, r \geq 0$ and $q$.

This model falls in the setting developed in Section \ref{subsec : linear_quadratic_control_problem} and applying Theorem \ref{thm : optimal_control_form} (to its extended version, see Remark \ref{rmk : cross_terms_control_state}), the optimal control $\hat{\alpha}$ is given by
\begin{align}
    \alpha_t^{\star} =  \frac{1}{2} \Big( q(X_t - \tilde{\E} \big[ \tilde{G}_q(U,\tilde{U}) \tilde{X}_t \big]) - Y_t \Big), \quad \forall t \in [0,T].
\end{align}
where $Y$ is given by \eqref{eq : form_adjoint_process} and where $K, \bar{K} , \Lambda$ are the solution to the associated Riccati equations.

We now present some numerical examples which illustrate the algorithms developed and their accuracy in learning the optimal trajectory $X^{\star}$ associated to the optimal control $\alpha^{\star}$ using $N=10000$ and $L=50$.
We plot it by approximating the optimal trajectory  with one and four  moments and show the results. The results are very stable with the number of moments taken but taking four moments for this simple interaction case slightly degrades the results.



\begin{figure}[H]
    \centering
    \begin{minipage}[t]{0.47\linewidth}
   
    \includegraphics[width=\linewidth]{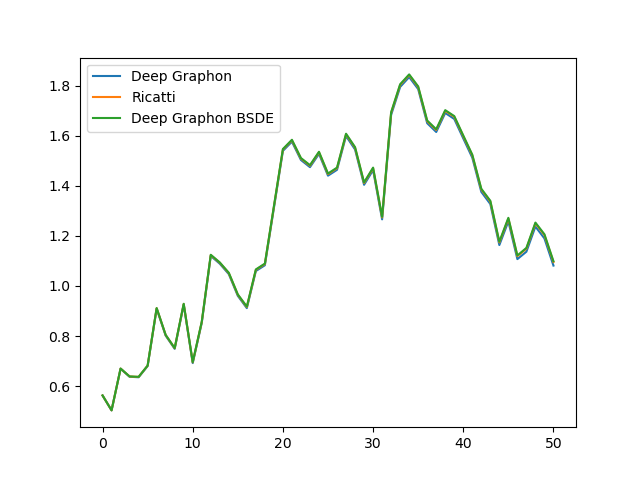}
    \caption*{Optimal trajectory of $(X^{\star}_t)_{0 \leq t \leq T}$ }
    \label{fig:enter-label}
    \end{minipage}
    \begin{minipage}[t]{0.47\linewidth}
 
   \includegraphics[width=\linewidth]{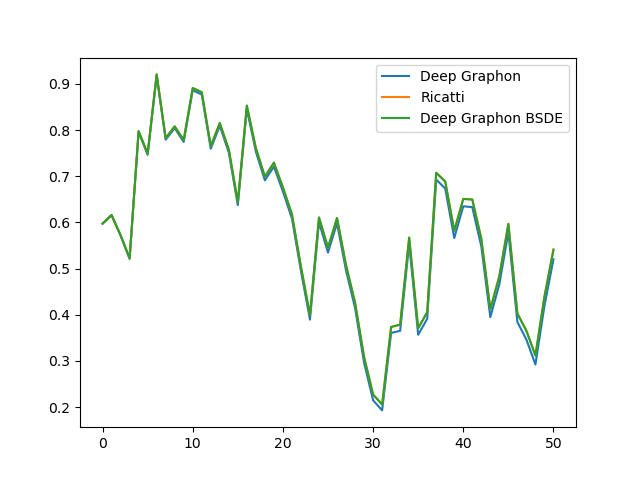}
    \caption*{ Optimal trajectory of $(X^{\star}_t)_{0 \leq t \leq T}$}
    \label{fig:enter-label}
    \end{minipage}
    \begin{minipage}[t]{0.47\linewidth}

    \includegraphics[width=\linewidth]{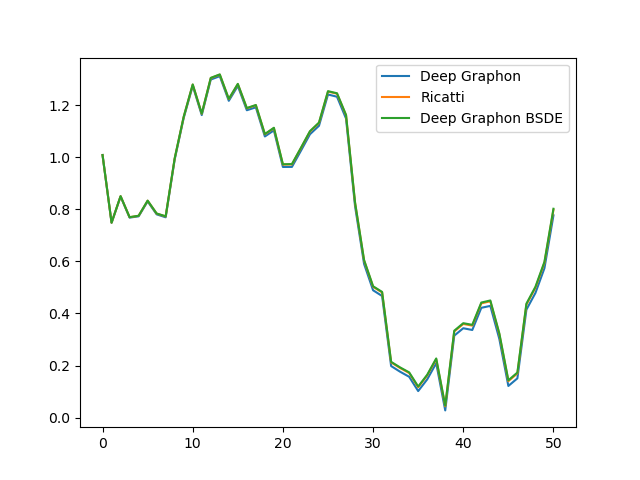}
    \caption*{Optimal trajectory of $(X^{\star}_t)_{0 \leq t \leq T}$}
    \label{fig:enter-label}
    \end{minipage}
     \begin{minipage}[t]{0.47\linewidth}
   \includegraphics[width=\linewidth]{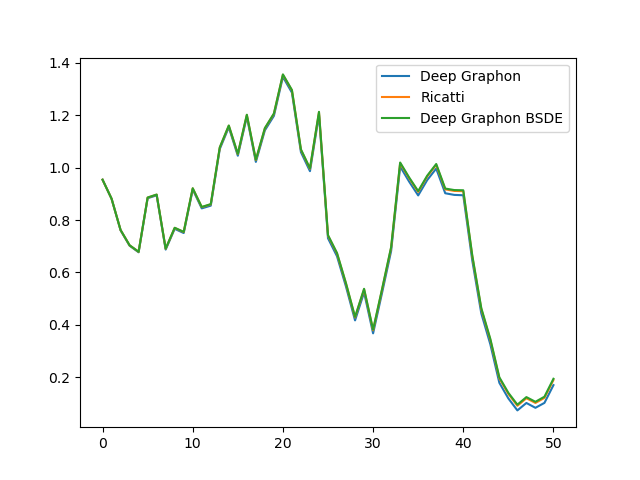}
    \caption*{Optimal trajectory of $(X^{\star}_t)_{0 \leq t \leq T}$}
    \label{fig:trajConMomentLQM1}
    \end{minipage}
    \caption{One moment approximation of  the optimal feedback map $a^{\star}$ and associated optimal state trajectory $X^{\star}$  (with the feedforward network): }
\end{figure}

\begin{figure}[H]
    \centering
    \begin{minipage}[t]{0.47\linewidth}
   
    \includegraphics[width=\linewidth]{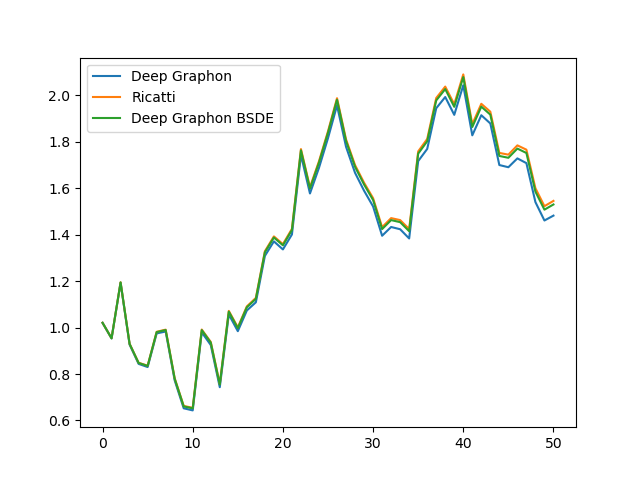}
    \caption*{Optimal trajectory of $(X^{\star}_t)_{0 \leq t \leq T}$ }
    \label{fig:enter-label}
    \end{minipage}
    \begin{minipage}[t]{0.47\linewidth}
 
   \includegraphics[width=\linewidth]{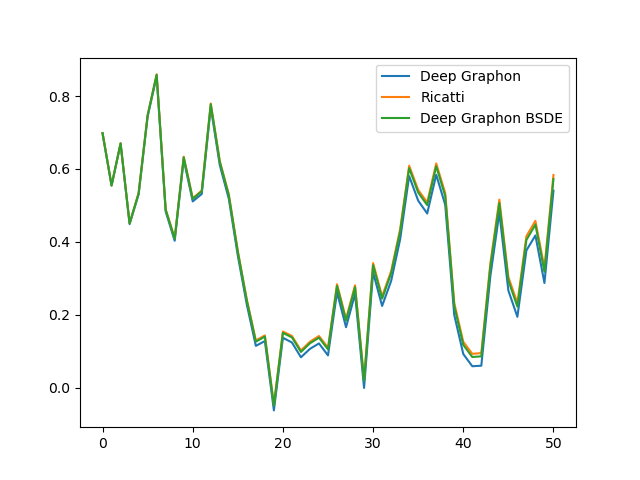}
    \caption*{ Optimal trajectory of $(X^{\star}_t)_{0 \leq t \leq T}$}
    \label{fig:enter-label}
    \end{minipage}
    \begin{minipage}[t]{0.47\linewidth}

    \includegraphics[width=\linewidth]{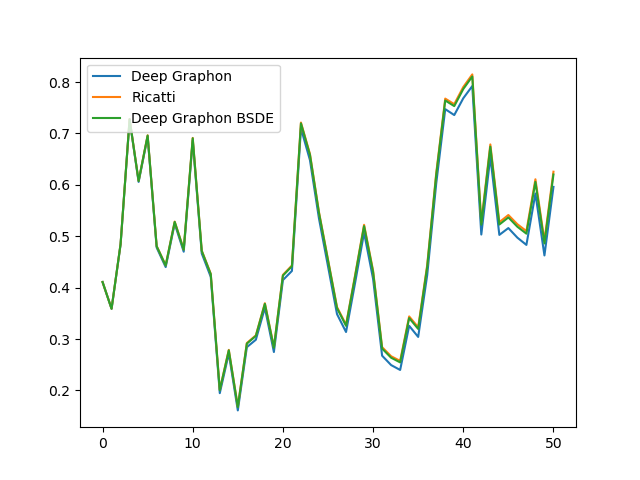}
    \caption*{Optimal trajectory of $(X^{\star}_t)_{0 \leq t \leq T}$}
    \label{fig:enter-label}
    \end{minipage}
     \begin{minipage}[t]{0.47\linewidth}
   \includegraphics[width=\linewidth]{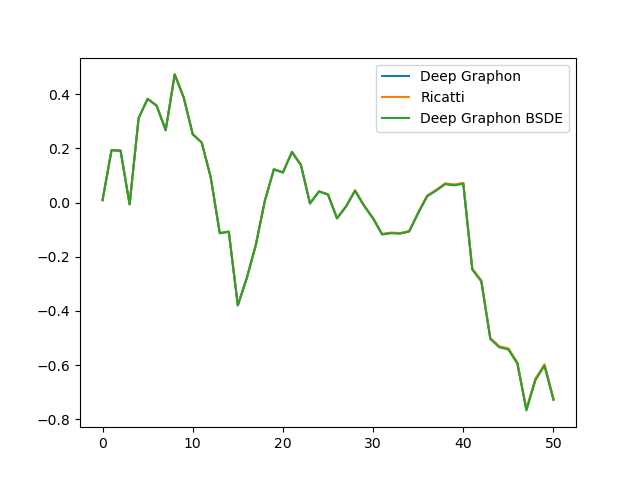}
    \caption*{Optimal trajectory of $(X^{\star}_t)_{0 \leq t \leq T}$}
    \label{fig:trajConMomentLQM1}
    \end{minipage}
\caption{Four moments approximation of  the optimal feedback map $a^{\star}$   (with the feedforward network)}
\end{figure}

\newpage

We also give below two tables summarizing the results obtained sampling 1000 distributions.
\begin{table}[H]
    \centering
     \begin{minipage}[t]{0.47\linewidth}
         \begin{tabular}{|c|c|c|c|c|}
       \hline        
       Criteria  &  $E^{\text{abs}}_{\text{Alg}}$&  $E^{L_2}_{\text{Alg}}$  & $E^{\text{sup}}_{\text{Alg}}$  & time\\
       $DG$  &  7.5e-4& 1.8e-6 &  7.9e-3   &14610\\
       $BSDE$ & 3e-4&  9.7e-7 &  1.5e-2  & 10660\\
       \hline 
    \end{tabular}
 
    \caption{One moment}
    \vspace{0.5cm}
      \end{minipage}
         \begin{minipage}[t]{0.47\linewidth}
    \begin{tabular}{|c|c|c|c|c|}
    \hline 
       Criteria  &  $E^{\text{abs}}_{\text{Alg}}$&  $E^{L_2}_{\text{Alg}}$  & $E^{\text{Sup}}_{\text{Alg}}$  & time\\
         $DG$  & 1e-3 & 3.4e-6 &  1e-2 & 14650\\
       $BSDE$ & 2.7e-4& 5e-7&  6e-3 &  10770\\
       \hline 
   \end{tabular}
    \caption{two moments}
    \vspace{0.5cm}
     \end{minipage}
           \begin{minipage}[t]{0.47\linewidth}
    \begin{tabular}{|c|c|c|c|c|}
    \hline 
       Criteria  &  $E^{abs}_X$&  $E^{L_2}_X$  & $E^{Sup}_X$  & time\\
       $DG$  & 3e-4  &  3.5e-7&  4.7e-3 & 14770\\
       $BSDE$ &  2.5e-4 &  4.4e-7 & 9.7e-3 & 10870\\
       \hline 
    \end{tabular}
    \caption{three moments}
    \vspace{0.5cm}
     \end{minipage}
          \begin{minipage}[t]{0.5\linewidth}
    \begin{tabular}{|c|c|c|c|c|}
    \hline 
       Criteria  &  $E^{abs}_X$&  $E^{L_2}_X$  & $E^{Sup}_X$  & time\\
       $DG$  & 9e-4 & 3.5e-6& 1.3e-2  & 14900\\
       $BSDE$ & 3e-4  & 2e-6  & 2.9e-2 & 11110 \\
       \hline 
    \end{tabular}
    \caption{four moments}
    \vspace{0.5cm}
     \end{minipage}
       \caption{ Error with Riccati as reference depending on the number of  moments $J$ used as inputs of neural networks. Feedforward with tangent 3 layers of 10 neurons}
   \label{tab:LQFeed}
\end{table}

\begin{table}[H]
    \centering
     \begin{minipage}[t]{0.49\linewidth}
         \begin{tabular}{|c|c|c|c|c|}
         \hline 
       Criteria  &  $E^{abs}_X$&  $E^{L_2}_X$  & $E^{Sup}_X$  & time\\
       $DG$  & 1e-4& 3.8e-8 &  8e-4   & 17630\\
       $BSDE$ & 4.5e-5&  2.8e-8 &  2.7e-3  & 13180\\
       \hline 
    \end{tabular}
 
    \caption{one moment}
    \vspace{0.5cm}
      \end{minipage}
         \begin{minipage}[t]{0.49\linewidth}
    \begin{tabular}{|c|c|c|c|c|}
    \hline 
       Criteria  &  $E^{abs}_X$&  $E^{L_2}_X$  & $E^{Sup}_X$  & time\\
         $DG$  & 2.8e-5 & 5.2e-9 &  8e-4 &  17690\\
       $BSDE$ & 5.4e-5& 2.2e-8&  1.6e-3 & 12850 \\
       \hline 
   \end{tabular}
    \caption{two moments}
    \vspace{0.5cm}
     \end{minipage}
           \begin{minipage}[t]{0.49\linewidth}
    \begin{tabular}{|c|c|c|c|c|}
    \hline 
       Criteria  &  $E^{abs}_X$&  $E^{L_2}_X$  & $E^{Sup}_X$  & time\\
       $DG$  &8.5e-5  &  2.1e-8&  7.e-4 & 17400\\
       $BSDE$ &  1.e-4 &  2e-7 & 6.6e-3 &  13060\\
       \hline 
    \end{tabular}
    \caption{three moments}
    \vspace{0.5cm}
     \end{minipage}
          \begin{minipage}[t]{0.5\linewidth}
    \begin{tabular}{|c|c|c|c|c|}
    \hline 
       Criteria  &  $E^{abs}_X$&  $E^{L_2}_X$  & $E^{Sup}_X$  & time\\
       $DG$  & 2.7e-4 & 2.1e-5& 1.4e-1  & 17740 \\
       $BSDE$ & 1.2e-4  & 6.9e-7  & 1.7e-2 & 13190 \\
       \hline 
    \end{tabular}
    \caption{four moments}
    \vspace{0.5cm}
     \end{minipage}
       \caption{ Error with Riccati as reference depending on the number of  moments $J$ used as inputs of neural networks. Spline KAN with 2 layers of 10 neurons, 5 grid meshes.}
   \label{tab:LQKAN}
\end{table}

Results in Tables \ref{tab:LQFeed} , \ref{tab:LQKAN} show the accuracy of the method. The KAN network allow us to get more accurate results. The use of a high number of moments degrades the results.

\subsubsection{A non linear quadratic example}

As an illustration of the algorithm above, we consider a one-dimensional model with 
\begin{align}
\begin{cases}
    b(t,u,x,\mu,a) &= \beta(t,u,x,\mu) + a, \\
    f(t,u,x,\mu,a) &= F(t,u,x,\mu) + \frac{a^2}{2}, \\
    g(u,x,\mu)&= \E_{(U,\xi) \sim \mu} \big[ w(x - G(u,U) \xi) \big],
\end{cases}
\end{align}
where $w$ is a map on $\Cc_2(\R)$, $F$ is a map on $I \times \R^d \times \Mc_{\lambda}$ to be chosen later and $G$ is a symmetric map over $I \times I$. In our setting, we are looking for a value function to the Bellman equation in the form $V(u,t,x,\mu) =  e^{\eta(T-t)} \E_{(U,\xi) \sim \mu} \big[ w(x-G(u,U) \xi) \big]$ for a positive constant $\eta \geq 0$ and by simple calculations, the master field map is given by 
\begin{align}\label{eq : master_field_map}
    \Uc(t,u,x,\mu) &=  e^{\eta(T-t)} \Big( \E_{(U,\xi) \sim \mu} \big[ w'(x- G(u,U) \xi) \big] - \E_{(U,\xi) \sim \mu} \big[ G(U,u) w'(\xi - G(U,u)x) \Big] \Big). 
\end{align}
Taking $w= \text{cos}$ and defining to alleviate notations the following quantities
\begin{align}
\begin{cases}
    A(u,\mu) &:= \E_{(U,\xi) \sim \mu} \big[ \text{cos}(G(u,U) \xi) \big], \\
    B(u,\mu) &:= \E_{(U,\xi) \sim \mu} \big[ \text{sin}(G(u,U) \xi) \big], \\
    k(u,v) &:=\eta + \frac{\sigma^2}{2} \big(1+ G^2(u,v) \big) \\
    C(u,x,\mu) &:= \E_{(U,\xi) \sim \mu} \big[    G(u,U) \text{sin}(\xi)\text{cos}(G(u,U)x) \big] \\
    D(u,x,\mu) &:=\E_{(U,\xi) \sim \mu} \big[    G(u,U) \text{cos}(\xi)\text{sin}(G(u,U)x) \big]  \\
    K_{\text{cos}}(u,\mu) &:= \E_{(U,\xi)\sim \mu} \big[k(u,U)\text{cos}(G(u,U)\xi) \big], \\
    K_{\text{sin}}(u,\mu) &:= \E_{(U,\xi)\sim \mu} \big[k(u,U)\text{sin}(G(u,U)\xi) \big],  \\
    B_{\text{cos}}(t,u,x,\mu) &:= \E_{(U,\xi) \sim \mu} \Big[  \big( G(u,U) \beta(t,U,\xi,\mu) - \beta(t,u,x,\mu) \big) \text{cos}(G(u,U)\xi)          \Big], \\
    B_{\text{sin}}(t,u,x,\mu) &:= \E_{(U,\xi) \sim \mu} \Big[  \big( G(u,U) \beta(t,U,\xi,\mu) - \beta(t,u,x,\mu) \big) \text{sin}(G(u,U)\xi)          \Big], \\
    U_{\text{cos}}(t,u,x,\mu) &:= \E_{(U,\xi) \sim \mu} \Big[  \big(  \Uc(t,u,x,\mu) - G(u,U)\Uc(t,U,\xi,\mu) \big) \text{cos}(G(u,U)\xi)          \Big]  , \\
    U_{\text{sin}}(t,u,x,\mu) &:= \E_{(U,\xi) \sim \mu} \Big[  \big(  \Uc(t,u,x,\mu) - G(u,U)\Uc(t,U,\xi,\mu) \big) \text{sin}(G(u,U)\xi)          \Big],
\end{cases}
\end{align}
choosing the map $F$ to be equal to 
\begin{align}
    F(t,u,x,\mu)=  e^{\eta(T-t)} \Big( \text{cos}(x) \alpha(t,u,x,\mu) + \text{sin}(x) \gamma(t,u,x,\mu) \Big) - \frac{1}{2} e^{2 \eta(T-t)} \Lambda(u,x,\mu)^2,
\end{align}
where we set
\begin{align}
\begin{cases}
     \alpha(t,u,x,\mu) &:= K_{\text{cos}}(u,\mu) + B_{\text{sin}}(t,u,x,\mu)  + U_{\text{sin}}(t,u,x,\mu), \\
    \gamma(t,u,x,\mu) &:= K_{\text{sin}}(u,\mu) - B_{\text{cos}}(t,u,x,\mu) - U_{\text{cos}}(t,u,x,\mu), \\
    \Lambda(u,x,\mu) &:= - \text{sin}(x) A(u,\mu) + \text{cos}(x) B(u,\mu) + C(u,x,\mu) - D(u,x,\mu), 
\end{cases}
\end{align}
the function $V$ defined above satisfies the Bellman equation. Moreover, the optimal control $a^{\star}$ is defined as 
\begin{align}
    a^{\star}(t,u,x,\mu) = - \partial_{\tilde{x}} \frac{\delta}{\delta m}v(t,\mu)(u,x), \quad \text{with } v(t,\mu) = \E_{(U,\xi) \sim \mu} \big[ V(t,U,\xi,\mu) \big]. 
\end{align}

We give below the results of the learning of the value function $V$  using $N=6000$ and $L=50$. We observe that the error is much important than in the LQ case.
This error is independent of the number $J$ of moment taken. Some numerical tries seem to indicate that the error is mainly due to the Euler scheme resolution and the use of KAN does not improve the results.

\begin{table}[H]
    \centering
    \begin{tabular}{|c|c|c|c|c|}
    \hline 
     NbMom &  $E^{abs}_\text{Alg}$&  $E^{L_2}_\text{Alg}$  & $E^{Sup}_\text{Alg}$ & time \\
     1   & 8.9e-3& 1.2e-4 & 4.4e-2 & 42600\\
     2   & 8.6e-3 & 1.1e-4 & 3.2e-2  & 42900\\
     3   & 8.7e-3 & 1.1e-4 &  3.4e-2 &  42800\\
     4   & 8.4e-3  & 1.2e-4 & 3.4e-2 & 42000    \\
     \hline 
 \end{tabular}
  \caption{Result on non LQ problem with interacting network with feed forward : 3 hidden layers with 10 neurons}
\end{table}

\appendix

\renewcommand{\thesection}{\Alph{section}}

\section{Analysis tools over the space $\Mc_{\lambda}$}\label{sec : analysis_Plambda}

In this Appendix, we present the main analysis tools used to handle maps defined over the space $\Mc_{\lambda}$. It essentially relies on the notion of linear functional derivative and/or equivalently on the Lion's derivative for which we recall a nice introduction in \cite{carmona_probabilistic_2018a} (Chapter 5). The notion of derivative and convexity we are going to introduce is fairly inspired from the works in \cite{kharroubi2025stochastic,decrescenzo2024mean} where the maps were instead defined over the space $L^2(I ; \Pc_2(\R^d))$. In fact, both notions of derivatives are fairly similar but for sake of completeness, we reintroduce it in full generality.
\subsection{A notion of derivative}\label{def :derivative_notion}

\begin{Definition}\textnormal{(Linear functional derivative on $\Mc_{\lambda}$).}\label{def : linear_functional_derivativePlambda}

   \begin{enumerate}
       \item     Given a function $v : \Mc_{\lambda}\to \R$, we say that a measurable function 
    \begin{align}
        \frac{\delta }{\delta m }v : \Mc_{\lambda} \times I \times \R^d \ni (\mu,u,x) \mapsto \frac{\delta }{\delta m }v(\mu)(u,x) \in \R,
    \end{align}
    is the linear functional derivative (or flat derivative) of $v$
    \begin{enumerate}
        \item [(1)] For every compact $K \in \Mc_{\lambda}$, there exists a compact $C_K > 0$ such that 
        \begin{align}
            \Big | \frac{\delta}{\delta m }v(\mu)(u,x) \Big | \leq C_K(1 + |x|^2 ),
        \end{align}
        for every $u \in I$, $x \in \R^d$ and $\mu \in K$.
        \item [(2)] For every $\mu,\nu \in \Mc_{\lambda}$, we have
        \begin{align}
            v(\nu) - v(\mu) &= \int_{0}^{1} \int_{I \times \R^d} \frac{\delta }{\delta m }v((1-\theta)\mu + \theta \nu)(u,x) \d (\nu-\mu)(u,x) \d \theta,\\
            &= \int_{0}^{1} \int_{I} \int_{\R^d} \frac{\delta }{\delta m }v((1-\theta)\mu + \theta \nu)(u,x) \d (\nu^u-\mu^u)(x) \lambda(\d u ) \d \theta
        \end{align}
    \end{enumerate}

       \item We  say that the function $v$ admits a continuously differentiable flat derivative if

\begin{itemize}
     \item[(1)] $v$ admits a flat derivative $\frac{\delta}{\delta m }v$ satisfying  $x \mapsto  \frac{\delta}{\delta m }v(\mu)(u,x)$ is Fréchet differentiable with Fréchet derivative denoted by $x \mapsto \partial_{\tilde{x}} \frac{\delta}{\delta m } v(\mu)(u,x)$ for all $(\mu,u) \in \Mc_{\lambda} \times I$
    \item [(2)] The map $(\mu,x) \mapsto \partial_{\tilde{x}} \frac{\delta}{\delta m }v(\mu)(u,x)$ is continuous from $\Mc_{\lambda} \times I $ into $\R^d$ for $\lambda-\text{a.e}$ $u \in I$.
    \item [(3)] For every compact set $K \subset \Mc_{\lambda}$, there exists a constant $C_K \geq 0$ such that
    \begin{align}
        \Big |\partial_{\tilde{x}} \frac{\delta}{\delta m }v(\mu)(u,x) \Big| \leq C_K \big( 1+ |x| \big),
    \end{align}
    for every $u \in I$, $x \in \R^d$ and $\mu \in K$.
\end{itemize}

\item We say that $v : [0,T] \times \Mc_{\lambda} \to \R$ is of class $\tilde{\Cc}^{1,2}([0,T] \times \Mc_{\lambda})$ if
\begin{enumerate}
    \item [(1)] For every $\mu \in \Mc_{\lambda}$, the map $t \mapsto v(t,\mu)$ is continuously differentiable on $[0,T]$ and we denote by $(t,\mu) \mapsto \partial_t v(t,\mu)$ its time derivative.
    \item [(2)] For every $t \in [0,T]$, the derivative $\frac{\delta}{\delta m}v(t,\mu)(u,x)$ exists and its measurable in all its arguments.
    \item [(3)] $\frac{\delta}{\delta m}v(t,\mu)(u,x)$ is twice continuously differentiable on $\R^d$, i.e for any $(t,\mu,u) \in [0,T] \times \Mc_{\lambda} \times I$, the map $x \mapsto \frac{\delta}{\delta m}v(t,\mu)(u,x)$ and the gradient and the Hessian matrix 
    \begin{align}
        \partial_{\tilde{x}} \frac{\delta}{\delta m}v : [0,T] \times \Mc_{\lambda} \times I \times \R^d  \to \R^d, \quad \partial^2_{\tilde{x}} \frac{\delta}{\delta m} v : [0,T] \times \Mc_{\lambda} \times I \times \R^d  \to \R^{d \times d}
    \end{align}
    satisfy the following growth conditions: There exists a positive constant $C$ s.t
    \begin{align}
    \begin{cases}
            \Big | \partial_{\tilde{x}}  \frac{\delta}{\delta m}v(t,\mu)(u,x) \Big | \leq C \big( 1+ |x| + \Wc(\mu, m \otimes \delta_0), \\
            \Big | \partial^2_{\tilde{x}}  \frac{\delta}{\delta m} v(t,\mu)(u,x) \Big |  \leq C,
        \end{cases}
    \end{align}
    for every $t \in [0,T]$, $u \in I$, $x \in \R^d$ and $\mu \in \Mc_{\lambda}$.
    \item [(4)] The map $[0,T] \times \Mc_{\lambda} \ni (t,\mu) \mapsto \partial_t v(t,\mu)$ is continuous.
    \item [(5)] For every $u \in I$ and every compact set $H \subset \R^d$, the functions $\partial_{\tilde{x}} \frac{\delta}{\delta m}v(t,\mu)(u,x)$ and $\partial^2_{\tilde{x}} \frac{\delta}{\delta m }v(t,\mu)(u,x)$ are continuous functions of $(t,\mu)  \in [0,T] \times \Mc_{\lambda}$ uniformly in $x \in H$.
\end{enumerate}
\end{enumerate}

\begin{Remark}
\normalfont 

\noindent 
\begin{enumerate}
    \item [$\bullet$]   In the core of the paper, the map $v$ will be mainly defined over the space $I \times \R^d \times \Mc_{\lambda}$. In this setting, the flat derivative of $v$ is defined as a measurable map 
    \begin{align}
        \frac{\delta }{\delta m }v : I \times \R^d \times \Mc_{\lambda} \times I \times \R^d \ni (u,x,\mu,\tilde{u},\tilde{x}) \mapsto \frac{\delta }{\delta m }v(u,x,\mu)(\tilde{u},\tilde{x}) \in \R,
    \end{align}
    and you can extend the previous points from Definition \eqref{def : linear_functional_derivativePlambda}.
    \item [$\bullet$] Some example of standard type of functions for which we can compute the linear functional derivatives are discussed in \cite{decrescenzo2024mean,kharroubi2025stochastic}.
\end{enumerate}

\end{Remark}

\end{Definition}

\subsection{A notion of convexity}

\begin{Definition}\textnormal{(Convexity on $\Mc_{\lambda}$).}

Given a function $v :\Mc_{\lambda} \to \R$ supposed to admit a continuously differentiable flat derivative in the sense of Definition \ref{def : linear_functional_derivativePlambda}, we say that $v$ is convex if for every $\mu,\mu' \in \Mc_{\lambda}$, we have
\begin{align}
    v(\mu') - v(\mu) \geq  \E \Big[ \partial_{\tilde{x}} \frac{\delta}{\delta m}v(\mu)(U,X) \cdot (X' - X) \Big],
\end{align}
where $(U,X) \sim \mu$ and $(U,X') \sim \mu'$.

More generally, if $v$ is now defined on $I \times \R^d \times \Mc_{\lambda}$, we say that $v$ is  said to be convex if for every $(x,\mu) \in \R^d \times \Mc_{\lambda}$ and $(x',\mu') \in \R^d \times \Mc_{\lambda}$ and for $\lambda(\d u )-a.e$, we have
\begin{align}
    v(u,x',\mu') - v(u,x,\mu) \geq  \partial_x v(u,x,\mu) \cdot (x'-x) + \E \Big[ \partial_{\tilde{x}} \frac{\delta}{\delta m} v(u,x,\mu)(U,X) \cdot (X' -X) \Big],
\end{align}
where $(U,X) \sim \mu$ and $(U,X') \sim \mu'$.
\end{Definition}

\section{Optimal control related results}

\subsection{Proof of Theorem \ref{thm : existence_unicity_process} }\label{subsec : prop_existenceunicitySDE}

 The proof follows from a standard  fixed point argument on the complete metric space $(\Mc_{\lambda}(I \times \Cc^d_{[t,T]}) ; \Wc_2)$ and is a straightforward adaptation of the proof of Theorem 2.6 in \cite{decrescenzo2024mean}. Indeed, define the map $\Phi$ as
    \begin{align}
        \Mc_{\lambda}(I \times \Cc^d_{[t,T]}) \ni \nu \mapsto \Phi(\nu)= \P_{(U,X^{\nu})},
    \end{align}
    where $X^{\nu}$ denotes the solution to the standard SDE with Lipschitz coefficients
    \begin{align}\label{eq : SDE_uncoupled}
    \begin{cases}
        \d X_s^{\nu} &= b(U,X_s^{\nu}, \nu_s, \alpha_s) \d s  + \sigma(U,X_s^{\nu}, \nu_s, \alpha_s) \d W_s, \\
        X_t^{\nu} &= \xi,
    \end{cases}
    \end{align}
    and where $\nu_s = x_s \sharp \nu$ with $x_s : I \times \Cc^d_{[t,T]} \ni (u,\omega) \mapsto x_s(u,\omega)=(u,\omega_s) \in I \times \R^d$ denotes the projection map. Under Assumption \ref{assumption : regularity_coefficients}, the SDE \eqref{eq : SDE_uncoupled} admits a continuous $\F$-adapted process and therefore $\P_{(U,X^{\nu})}$ can be viewed as a probability measure on $\Pc^{\lambda}(I \times \Cc^d_{[t,T]})$. Moreover,under standard estimates it is easy to verify that $\E \Big[ \underset{0 \leq t \leq T}{\text{ sup}} | X_t^{\nu}|^2 \Big] < \infty$ such that $\P_{(U,X^{\nu})} \in \Mc_{\lambda}(I \times \Cc^d_{[t,T]})$ and therefore $\Phi : \Mc_{\lambda}(I \times \Cc^d_{[t,T]}) \to \Mc_{\lambda}(I \times \Cc^d_{[t,T]})$ is well defined. Now, from standard estimates, we prove that $\Phi$ has a unique fixed point $\bar{\nu}$ in $\Mc_{\lambda}(I \times \Cc^d_{[t,T]})$. Indeed, we have (assuming for simplicity $b=0$)
    \begin{align}
        \E \Big[ \underset{t \leq s \leq r}{\text{ sup}} |X_s^{\nu} - X_s^{\mu}|^2 \Big] &\leq C \E \Big[ \int_{t}^{r} |\sigma(U,X_s^{\nu}, \nu_s, \alpha_s) - \sigma(U,X_s^{\mu},\mu_s,\alpha_s)|^2 \d s \Big] \\
        &\leq  C \E \Big[ \int_{t}^{r} \big(   \E \big[ |X_s^{\nu} - X_s^{\mu}|^2 \big] + \Wc_2(\nu_s,\mu_s)^2 \big) \d s \\
        &\leq C \E \Big[ \int_{t}^{r} \big(\E \big[ \underset{t \leq q \leq s}{\text{ sup}} |X_q^{\nu}-X_q^{\mu}|^2 \big] + \Wc_2(\nu_s,\mu_s)^2 \big) \d s,
    \end{align}
    where $C$ is a positive constant which can change from line to line. From Grönwall's lemma, we end up with
    \begin{align}
        \Wc_2(\Phi(\nu),\Phi(\mu) )^2 \leq  \E \Big[ \underset{t \leq s \leq T}{\text{ sup}} |X_s^{\nu} - X_s^{\mu}|^2 \Big] \leq C(T-t) \Wc_2(\nu,\mu)^2,
    \end{align}
    and where we  used $\Wc_2(\nu_s,\mu_s) \leq \Wc_2(\nu,\mu)$ for any $t \leq s \leq T$. Now, from standard arguments we conclude that the sequence $(\nu^{(k)})_{k \in \N}$ defined as
    \begin{align}
    \begin{cases}
        \nu^{(k+1)} = \Psi(\nu^{(k)}), \\
        \nu^{(0)} \text{ arbitrary point in $\Mc_{\lambda}(I \times \Cc^d_{[t,T]})$.}
    \end{cases}
    \end{align}
    is a Cauchy sequence for $\Wc_2$ and converges on $\Mc_{\lambda}(I \times \Cc^d_{[t,T]})$ towards a unique fixed point $\bar{\nu}$. Considering the associated process $X^{\bar{\nu}}$ yields the unique solution to Theorem \ref{thm : existence_unicity_process}.

\subsection{Proof of Proposition \ref{prop : pontryagin_formulation}}\label{appendix : proof_optimality_alpha}

The proof is essentially an adaption of the proof in \cite{kharroubi2025stochastic} and therefore we  just give the main ideas.

\vspace{1mm}

\noindent \textbf{Step n°1 : Definition of the variation process.}

\noindent Given an admissible control $\beta = (\beta_t)_{0 \leq t \leq T} \in \Ac$, we define the control $\delta := \beta - \alpha$. We notice that $\delta \in \Ac$ since $A$ is a convex set.  We now define the $\R^d$-valued variation process $V=(V_t)_{0 \leq t \leq T}$ associated to the process $X$ defined in \eqref{eq : SDE_label_state} as the solution to the following SDE
\begin{align}
\begin{cases}
    \d V_t &= \bigg[ \gamma_t \cdot V_t + \tilde{\E} \Big[ \partial_{\tilde{x} } \frac{\delta}{\delta m } b(U,X_t, \P_{(U,X_t)},\alpha_t)(\tilde{U},\tilde{X}_t) \tilde{V}_t \Big] + \eta_t \cdot \delta_t \bigg] \d t \\
    &\quad + \bigg[ \hat{\gamma}_t \cdot V_t + \tilde{\E} \Big[ \partial_{\tilde{x} } \frac{\delta}{\delta m } b(U,X_t, \P_{(U,X_t)},\alpha_t)(\tilde{U},\tilde{X}_t) \tilde{V}_t \Big] + \hat{\eta}_t \cdot \delta_t \bigg] \d W_t, \\
    V_0 &=0,
\end{cases}
\end{align}
and where we denoted for every $t \in [0,T]$
\begin{align}
\begin{cases}
    \gamma_t &= \partial_x b(U,X_t,\P_{(U,X_t)},\alpha_t), \quad \eta_t = \partial_a b(U,X_t,\P_{(U,X_t)},\alpha_t), \\
    \hat{\gamma}_t &= \partial_x \sigma(U,X_t,\P_{(U,X_t)},\alpha_t), \quad \hat{\eta}_t =  \partial_{a} \sigma(U,X_t,\P_{(U,X_t)},\alpha_t).
\end{cases}
\end{align}
We note that under the Assumptions \ref{assumption : regularity_coefficients}-\ref{assumption2: regularity_coefficients} that $(X,V)$ is an SDE satisfying the Assumptions of Theorem \ref{thm : existence_unicity_process} and hence, $(X,V)$ and hence $V$ is uniquely defined.

We now denote the family of admissible controls $\alpha^{\epsilon} = \alpha + \epsilon \delta$ for $\epsilon \in [0,1]$ and by $X^{\epsilon} := X^{\alpha^{\epsilon}}$ the associated controlled state process. Then, one can check under Assumptions \ref{assumption2: regularity_coefficients}.
\begin{align}\label{eq : variation_process_definition}
    \E \Big[ \underset{0 \leq t \leq T}{\text{sup }} |\frac{X_t^{\epsilon} - X_t}{\epsilon} - V_t |^2 \Big] \underset{\epsilon \to 0}{\to } 0.
\end{align}

\noindent \textbf{Step n°2 : Gâteaux derivative of $J$.}

\noindent Now, relying in \eqref{eq : variation_process_definition}, on Assumptions \ref{eq : variation_process_definition} and on Definition of the adjoint process $Y$ in \eqref{eq : FBSDE_system_weak_formulation}, one can check after some straightforward computations that
\begin{align}\label{eq : gateaux_differentiable_J}
    \underset{\epsilon \to 0}{\text{ lim} }  \frac{1}{\epsilon} \Big( J(\alpha + \epsilon(\beta - \alpha)) - J(\alpha) \Big)  =\E \Big[\int_{0}^{T}  \partial_{a} H((U,X_t,\P_{(U,X_t)},Y_t,Z_t)  \cdot (\beta_t - \alpha_t) \d t \Big].
\end{align}
Now relying on the convexity assumption of the map $a \mapsto H(u,x,\mu,y,z,a)$ for any $(u,x,\mu,y,z) \in I \times \R^d \times \Pc_2^{\lambda}(I \times \R^d) \times \R^d \times \R^{d \times n}$, one can derive the statement of Proposition \ref{prop : pontryagin_formulation} from the Gateaux derivatives of $J$ in \eqref{eq : gateaux_differentiable_J}

\subsection{Additional notations for the linear quadratic case}\label{subsec:  notations_lq_example}

We work under the complete Hilbert space $L^2(I \times I ;\R^{d \times d})$ and $L^2(I;\R^d)$ and we introduce the  following notations
    \begin{enumerate}
      \item We say that a kernel $K \in L^2(I \times I ; \R^{d \times d})$ is symmetric if the following holds
      \begin{align}\label{eq : symmetric_kernels}
    K(v,u)^{\top} = K(u,v), \quad \text{ $\d u \otimes \d v$ a.e.} 
    \end{align}
We denote by $L^2_{\text{sym}}(I \times I ;\R^{d \times d})$ the space of kernels $K \in L^2(I \times I; \R^{d \times d})$ satisfying \eqref{eq : symmetric_kernels}.
    \item  Given a kernel $K \in L^2(I \times I ; \R^{d \times d})$,  its associated linear integral Hilbert-Schimdt operator $T_K$ is defined as
\begin{align}\label{eq : operator_definition}
    L^2(I;\R^{d}) \ni f \mapsto T_K(f)(\cdot) = \int_I K(\cdot,v) f(v) \d v  \in L^2(I;\R^{d}).
\end{align}
We say that $T_K$ is a symmetric non-negative  operator on $L^2(I;\R^d)$ if its associated kernel satisfies \eqref{eq : symmetric_kernels} and $\langle f, T_Kf \rangle_{L^2(I;\R^d)} \geq 0$ for every $f \in L^2(I;\R^d)$.
    \item Given a kernel $K \in L^2(I \times I ; \R^{d \times d})$, we define  $K^* \in L^2(I \times I ; \R^{d \times d})$ as
   \begin{align}\label{eq : tilde_kernel}
    K^*(u,v): = K(v,u)^{\top}.
    \end{align}
When $K \in L^2_{\text{sym}}(I \times I ; \R^{d \times d})$, we have $K^*= K$. We also notice that $T_{K^*} = (T_K)^*$ where $(T_K)^*$ denotes the adjoint operator of $T_K$.
  \item Given two  kernels $K,W \in L^2(I \times I ; \R^{d \times d})$, it is easy to check that the  operator $T_{K} \circ T_{W}$ is associated with the kernel $\big(K \circ W \big)$ defined by
  \begin{align}\label{eq : kernel_product}
      \big(K \circ W \big)(u,v):= \int_I  K(u,w)W(w,v) \d w,
  \end{align}
   \item Given $L \in L^{\infty}(I ; \R^{d \times d})$, we define the multiplicative operator associated to $L$ as the linear operator on $L^2(I ; \R^{d})$ defined by 
  \begin{align}\label{eq : multiplicative_operator_definition}
    L^2(I;\R^{d}) \ni f \mapsto \Big( I \ni w \mapsto M_L(f)(w) = L^{w}f(w) \in \R^d \Big) \in L^2(I;\R^d).
\end{align}
Given $K,W \in L^2(I \times I ;\R^{d \times d})$ and $L \in L^{\infty}(I ;\R^{d \times d})$, the operator $T_K \circ M_L \circ T_W$ is associated with the kernel $\big(K \circ L \circ W \big)$ defined by
  \begin{align}\label{eq : kernel_form_multiplicativity}
    \big(K \circ L \circ W \big)(u,v) \ = \int_I K(u,w) L(w) W(w,v) \d w  .
\end{align}
    \end{enumerate}

 Given $(E, d)$ a Polish metric space (e.g. $\R^d$, $\Pc_2(\R^d)$) endowed with its natural Borel $\sigma$-algebra., we also introduce the following spaces.
   \begin{enumerate}
       \item [(1)]
       \begin{align}
        \begin{cases}
            L^2(I;E)  := \big \lbrace \varphi : I \ni u \mapsto \varphi(u) \text{ measurable  and $\int_I  d(\varphi(u), e)^2 \d u < + \infty \big \rbrace$}, \\
            L^2(I \times I; E)) := \big \lbrace  \varphi : I \times I \ni (u,v) \mapsto \varphi(u,v) \text{ measurable and } \int_{I \times I} \d (\phi(u,v),e)^2 \d u < + \infty \big \rbrace.  
        \end{cases}
       \end{align}
    \end{enumerate}
      Moreover, if $E$ is a normed metric space with norm $(\lVert \cdot \rVert)$ $(\text{e.g } \R^d$), we introduce
    \begin{enumerate}
       \item [(2)]
       \begin{align}
        \begin{cases}
            L^{\infty}(I;E) := \big \lbrace  \varphi : I \ni u \mapsto \varphi(u) \text{ measurable and $\underset{u \in I}{\text{ess sup }} \lVert \varphi(u) \rVert_{E} < + \infty. \big \rbrace$ },  \\
            L^{\infty}(I \times I ;E) :=  \big \lbrace  \varphi : I \times I \ni (u,v) \mapsto \varphi(u,v) \text{ measurable and $\underset{u,v \in I}{\text{ess sup }} \lVert \varphi(u,v) \rVert_{E} < + \infty. \big \rbrace$ }
        \end{cases}
        \end{align}
   \end{enumerate}
\subsection{Proof of Theorem \ref{thm : optimal_control_form}}\label{proof: thm_optimal_control_form}

The proof is essentially a combination of the proofs in \cite{kharroubi2025stochastic,de2025linear} and therefore we just give the main ideas of the proofs in the current setting.

\vspace{1mm}

\noindent \textbf{Step n°1 : Ansatz form for $Y$.}

\noindent Motivated by the standard empirical link between the value function and the adjoint process $Y$ in the stochastic maximum principle and the linear-quadratic parametrization, we are looking for a solution to the FBSDE \eqref{eq : FBSDE_linear_quadratic} by guessing 
\begin{align}\label{eq : ansatz_Y}
    Y_t = K_t(U) X_t + \tilde{\E} \big[ \bar{K}_t(U,\tilde{U}) \tilde{X}_t \big] + \Lambda_t, \quad 0 \leq t \leq T,
\end{align}
where $K \in \Cc^1\big([0,T]; L^{\infty}(I;\S^d_+) \big)$, $\bar{K} \in \Cc^1([0,T]; L^2_{\text{sym}}\big(I \times I ; \R^d \times \R^d) \big)$ and $\Lambda \in \Cc^1([0,T] ;L^2(I;\R^d) \big)$ are to be determined through Riccati equations. Plugging the ansatz \eqref{eq : ansatz_Y} into the dynamics of $X$ and after some tedious but straightforward computations, we obtain the Riccati equations stated in Theorem \ref{thm : optimal_control_form}.

\vspace{1mm}

\noindent \textbf{Step n°2 : Solvability of the Riccati equations}

\noindent The solvability of $(K,\bar{K},\Lambda)$, i.e., the proof of their existence and unicity over respectively the spaces $\Cc^1 \big([0,T];L^{\infty}(I;\S^d_+)\big)$, $\Cc^1 \big([0,T]; L^2_{\text{sym}}(I \times I; \R^{d \times d}) \big)$ and $\Cc^1([0,T]; L^2(I;\R^d) \big)$ is done in \cite{de2025linear} (see \cite{de2025optimal} for an extension to the common noise setting). Essentially, the Riccati on $K$ is standard and can be solved by standard theory. However, the Riccati equation for $\bar{K}$ is of a new type due to the heterogenous structure of interactions and its solvability requires explicitly the assumptions on the model coefficients resulting in the positivity of $J$ (see \eqref{eq : positivity_cost_functional}). Indeed, we get a fundamental relation (see  Proposition 3.1 in \cite{de2025linear}) and we are able to bound the operator norm of $T_{\bar{K}_t}$ uniformly in time which gives us an a-priori estimate on the solution and helps us to go backward in time and get global existence and uniqueness over $[0,T]$. Finally, $\Lambda$ is a linear ODE on the Hilbert space $L^2(I;\R^d)$ and can be solved froms standard theory on infinite dimensional spaces.

Therefore, we identified a solution $(X,Y,Z)$ to the FBSDE. This solution is in fact unique and the unicity is proved in \cite{kharroubi2025stochastic} so we refer to the references therein for a full proof of this result.

 \newpage

\begin{small}
\bibliographystyle{plain}   
\bibliography{References.bib}   
\end{small}

\end{document}